\documentclass[11pt,reqno,a4paper]{amsart}

\usepackage{xcolor}
\usepackage{ulem}
\usepackage{amsmath,amssymb,epsfig,graphics,hyperref}

\newtheorem{theorem}{Theorem}[section]
\newtheorem{proposition}[theorem]{Proposition}
\newtheorem{hypothesis}[theorem]{Hypothesis}

\newtheorem{corollary}[theorem]{Corollary}
\newtheorem{lemma}[theorem]{Lemma}
\newtheorem{sub-lemma}[theorem]{Sub-Lemma}

\newtheorem{remark}[theorem]{Remark}

\def\Z{\mathcal{Z}}

\def\N{\mathcal{N}}

\def\NN{\mathbb{N}}

\def\RR{\mathbb{R}}

\def\ZZ{\mathbb{Z}}

\DeclareMathOperator{\var}{var}

\let\eps=\varepsilon

\def\Z{\mathcal{Z}}
\def\RR{{\mathbb R}}
\def\1{{{\mathit 1} \!\!\>\!\! I} }

\renewcommand{\limsup}{\mathop{{\overline {\hbox{{\rm lim}}}}}}

\newcommand{\cvto}[1]{\underset{#1}{\longrightarrow}}

\begin{document}

\title[Quantitative recurrence for $T,T^{-1}$ transformations]{Quantitative recurrence for $T,T^{-1}$ transformations}
\author{Fran\c{c}oise P\`ene \and Beno\^\i t Saussol}
\address{Univ Brest, Universit\'e de Brest, LMBA, Laboratoire de
Math\'ematiques de Bretagne Atlantique, CNRS UMR 6205, Brest, France}
\email{francoise.pene@univ-brest.fr}
\address{Aix Marseille Univ, CNRS, I2M, Marseille, France}
\email{benoit.saussol@univ-brest.fr}
\begin{abstract}
We are interested in the study of the asymptotic behaviour of return times in small balls for the $T,T^{-1}$-transformation. We exhibit different asymptotic behaviour (different scaling, different limit point process) depending on the respective dimensions of the measures of the two underlying dynamical systems. 
It behaves either as for the direct product of the underlying systems, or as for the $\mathbb Z$-extension of the driving system (also studied in this article), or as a more sophisticated process.
\end{abstract}
\date{\today}
\maketitle
\bibliographystyle{plain}
\tableofcontents

\section{Introduction}
Within the context of dynamical systems, quantitative recurrence forms a specific family of limit theorems where one either wants to precise the distribution of entrance times to certain regions of the phase space, or also to compute the time needed for an orbit to come back close to its starting point. It has been studied mainly for finite measure preserving transformations with good mixing properties.
The typical situation is to obtain an exponential law in the first case and a recurrence rate equal to the dimension of the measure in the second case.
 We refer to the book \cite{book} and references therein for a survey of these results and the relation with extreme value theory.

In this work we consider a map which preserves a finite measure, but its mixing  is not strong enough to be treated by classical methods. Indeed its behavior has a lot to do with an underlying deterministic random walk, an infinite measure preserving dynamical system. The quantitative recurrence in the infinite measure case has been studied, among the few works, by Bressaud, Zweimüller and the authors in \cite{bz,ps,PSZ1,PSZ2}.

More precisely, we study the particular case of the generalized $T,T^{-1}$-transformation, which is known, since~\cite[p. 682, Problem 2]{Weiss} and~\cite{Kalikow}, to be Kolmogorov but not loosely Bernoulli.
We will see that, depending on the measure dimensions
of both dynamical systems defining the $T,T^{-1}$-transformation, the quantitative recurrence properties are either analogous to those of mixing subshift of finite type, or to
those of an infinite measure preserving dynamical systems ($\mathbb Z$-extension of a subshift of finite type), or some more elaborate compound process.

Let us recall the definition of the generalized $T,T^{-1}$-transformation.
Let $(X,f,\mu)$ and $(Y,g,\nu)$ be two ergodic probability measure preserving dynamical systems, where $f$ (resp. $g$) is a transformation acting on a relatively compact metric spaces $X$ (resp.  $Y$), with $g$ invertible. For simplicity, we focus in this work on the case where $(X,f,\mu)$ is a mixing subshift of finite type endowed with an equilibrium state of a Hölder potential.
 We endow the product space $Z:=X\times Y$ with the product (sup) metric and denote all these metrics by $d$.
Let $h\colon X\to\ZZ$ be a measurable centered function $\int_Yh\,  d\mu=0$. 
We define the generalized $T,T^{-1}$ transformation by
\[
F(x,y) = \left( f (x),g^{h(x)}(y) \right).
\] 
This map preserves the product probability measure $\rho:=\mu\otimes\nu$.
We denote by $h_n(x)=h(x)+\cdots+h(f^{n-1}x)$. Note that
\begin{equation}\label{formuleFn}
F^n(x,y)=\left(f^n(x),g^{h_n(x)}(y)\right)\, .
\end{equation}
Our goal is to study fine (quantitative) recurrence properties of $F$, and more precisely to study the return times of the orbit $(F^n(x,y))_n$ in a ball $B_r^Z(x,y)=B_r^X(x)\times B_r^Y(y)$ around the initial point.
To describe our results we suppose in this introduction that the system $(Y,g,\nu)$ is also a mixing subshift of finite type endowed with an equilibrium measure of a Hölder potential.

We will prove in Section~\ref{secrecrate} that the recurrence rate is given by $\min(2d_\mu,d_\rho)$, where $d_m$ stands for the dimension of a measure $m$. 
We establish results of convergence in distribution in Section~\ref{seccvgcedistrib}.
In the particular case where $d_\mu<d_\nu$, we show in Section~\ref{sec:nullmu} how the first return time for $F$ coincide with the first return time for the $\mathbb Z$-extension of $(X,f,\mu)$ by the cocycle $h$ (dynamical system preserving an infinite measure), that has been studied by Yassine in \cite{yassine1,yassine2}. This $\mathbb Z$-extension consists in the dynamical system $(X\times\mathbb Z,\widetilde F,\mu\otimes \mathfrak m)$ with
\begin{equation}\label{eq:Zextension}
\widetilde  F\colon X\times\ZZ\circlearrowleft,\quad \widetilde  F(x,q) = \left( f (x), q+h(x)\right)\, ,
\end{equation}
so that
\[
\forall n\in\mathbb N^*,\quad \widetilde F^n(x,q)=\left(f^n(x),q+h_n(x)\right)\, ,
\]
and where $\mathfrak m$ denotes the counting measure on $\mathbb Z$. 
Then, in Section~\ref{sec:musinu}, we study the return time point process of $F$  with the time normalization $\max((\mu(B_r^X(x)))^2,\rho(B_r^Z(x,y)))$ and establish the convergence of this point process to:
\begin{itemize}
	\item a standard Poisson process if $d_\mu>d_\nu$, as if $F$ was the direct product $f\otimes g:(x,y)\mapsto(f(x),g(y)) $;
	\item a standard Poisson process taken at the local time at 0 of the limit Brownian motion $B$ of $(h_{\lfloor nt\rfloor}/\sqrt n)_n$ if $d_\mu<d_\nu$, as for the $\mathbb Z$-extension $\widetilde F$ (the result for $\widetilde F$ will be proved in Section~\ref{secZextension});
	\item a sum of Poisson processes of parameter $a$ taken at the local time of $B$ at some random points given by an independent Poisson process of parameter $b$,  if $d_\mu=d_\nu$ (the couple of parameters $(a,b)$ may be random, its distribution can then be explicitely computed).
\end{itemize}

In Section~\ref{secappmom} we prove Theorem~\ref{THMcvloi}, the core approximation of the moments of our hitting point process.
In Section~\ref{secZextension} we prove the convergence, stated in Theorem~\ref{ThmZextension}, of the point process of the $\mathbb Z$-extension $\widetilde F$.

Finally the appendix contains results about the moments of the limiting processes discussed above.

\section{Recurrence rate}\label{secrecrate}

We define, for $z=(x,y)\in Z$ and $r>0$, the first return time $\tau_r=\tau_r^F$ in the $r$-th neighbourhood of the initial point, i.e.
\[
\tau_r(z)=\tau_r^F(z):=\inf\{n\ge1\colon d(F^n(z),z)<r\}.
\]
The first quantity we want to consider is the pointwise recurrence rate defined by
\[
R(z)=R^F(z)=\lim_{r\to0} \frac{\log{\tau_r(z)}}{-\log r}\, ,
\]
when it exists, otherwise we define $\overline{R}$ with a limsup and $\overline{R}$ with a liminf.
We define the pointwise dimension $d_\mu(x)$ (resp. $d_\nu(y)$) of $\mu$ at $x$ (resp. $\nu$ at $y$) as
\[
d_\mu(x):=\lim_{r\rightarrow 0}\frac{\log(\mu(B_r^X(x)))}{-\log r}\quad\mbox{and}\quad
d_\nu(y):=\lim_{r\rightarrow 0}\frac{\log(\nu(B_r^Y(y)))}{-\log r}\, ,
\]
setting $B_r^X(x)$ and $B_r^Y(y)$ for the balls of radius $r$ respectively around $x$ in $X$ and around $y$ in $Y$. In this paper we assume that the pointwise dimensions exist a.e. and are constant, in particular they are equal then to the Hausdorff dimension of the measures $d_\mu$ and $d_\nu$.

It was proven in \cite{RS} that the upper recurrence rate is bounded from above by the pointwise dimension. Namely, for $\rho$-a.e. $z=(x,y)$ one has
\begin{equation}\label{majo1R}
\overline R(z)=\overline R^F(z)=\limsup_{r\to0} \frac{\log{\tau_r(z)}}{-\log r}\le d_\rho = d_\mu+d_\nu\, ,
\end{equation}

We write $\tau_r^f$ and $\tau_r^g$ for the respective first return times of $f$ and $g$ in the ball of radius $r$ around the original point. 
With these notations, the ball  $B_r^Z(x,y)$ of radius $r$ around $(x,y)$ in $Z$ is $B_r^X(x)\times B_r^Y(y)$ and 
\[
\tau_r(x,y)=\inf\left\{n\ge 1\, :\, f^n(x)\in B_r^X(x),\ g^{h_n(x)}(y)\in  B_y^Y(y)\right\}\, .
\]
Whenever the random walk $h_n$ returns to the origin it produces an exact return for the second coordinate since $g^0=id$. 
Therefore the study of the recurrence of the whole system may be estimated via the one of the $\ZZ$-extension $(X\times\mathbb Z,\widetilde F,\mu\otimes\mathfrak m)$. 
Indeed,\footnote{noticing that $\tau_r^{\widetilde F}(x,q)$ does not depend on $q\in\mathbb Z$} one has for any $(x,y)\in Z$ and any $r<1$,
\[
\tau_r(x,y)\le \tau_r^{\widetilde  F}(x)=\inf\left\{n\ge 1\, :\, f^n(x)\in B_r^X(x),\ h_n(x)=0\right\}\, ,
\]

which immediately gives the 
\begin{proposition}\label{rzext} The upper recurrence rate for $F$ is bounded from above by the one of the $\ZZ$-extension $\widetilde  F$
	\[
	\overline	R^F(x,y):= \limsup_{r\to0} \frac{\log{\tau_r(z)}}{-\log r} \le \overline R^{\widetilde  F}(x,0):=\limsup_{r\to0} \frac{\log{\tau_r^{\widetilde F}(x,0)}}{-\log r} \quad \text{for $\mu$-a.e. $x$ and any $y$}. 
	\]
\end{proposition}
The latter was studied by Yassine \cite{yassine1}. She proved that when $X$ is a mixing subshift of finite type endowed with an equilibrium measure associated to an H\"older potential and when $h$
is continuous\footnote{Since $h$ has integer values, this implies  that $h$ is locally constant and takes a finite number of values.}, then for $\mu$-a.e. $x\in X$
\begin{equation}\label{zext}
\lim_{r\to0}\frac{\log\tau_r^{\widetilde  F}(x,0)}{-\log r} = 2d_\mu.
\end{equation}
Combining Proposition~\ref{rzext} and~\eqref{zext} we obtain in this setting that for $\rho$-a.e. $z$
\begin{equation}\label{majo2R}
\overline{R}^F(z)\le 2d_\mu.
\end{equation}
When $d_\mu<d_\nu$, and under some hypotheses that we will describe later, this bound is optimal. However, the returns to the origin of the random walk $h_n$ are quite sparse, and the first return of the whole system may happen much before. This is what happens when $d_\mu>d_\nu$ where \eqref{majo1R} becomes optimal, again under the same hypotheses that we describe now.

Throughout the rest of this section we will make the following assumption.
\begin{hypothesis}\label{HHH}
	\begin{itemize}
		\item[(A)] The system $(X,f,\mu)$ 
		is a one-sided
		mixing subshift of finite type with finite alphabet $\mathcal A$, endowed with and  equilibrium measure $\mu$ with respect to some
		H\"older potential.
		\item[(B)] The system $(Y,g,\nu)$ is a two-sided
		mixing subshift of finite type with finite alphabet $\mathcal A'$, endowed with an equilibrium measure $\mu$ with respect to some
		H\"older potential, or more generally it has super-polynomial decay of correlations on Lipschitz functions\footnote{Actually, we just use the fact that $\underline R^g(y)=\underline R^{(g^{-1})}(y)=d_\nu$ for $\nu$-almost every $y\in Y$ and that there exists $K>0$ and $r_1>0$ such that, for any $r\in]0;r_1[$, any $y\in Y$ and any $k\ge r^{-d_\nu+\varepsilon}$, $\nu\left(B_{2r}^Y(y)\cap g^{-k}(B_{2r}^Y(y))\right)\le K\nu(B_r^Y(y))^2$.} and $Y$ has finite covering dimension\footnote{meaning that there exists $M$ such that for each $r>0$ there exists a cover of $Y$ by $r$-balls with multiplicity at most $M$, e.g. $Y$ is a subset of euclidean space.}  as in \cite{RS}.
		\item[(C)]The step function $h$ is Lipschitz and $\mu$-centered.
	\end{itemize}
\end{hypothesis} 

Let $\lambda_X>0$.
We use the notations $C_m(x)$ for the cylinder\footnote{$C_m(x)$ is here the set of $x=(x'_k)_{k\in\mathbb Z}$ such that $x'_k=x_k$ for all $k=0,...,m$.}
of generation $m$ (also called $m$-cylinder) containing $x$ and $\mathcal C_m$ for the set of the $m$-cylinders of $X$. We endow $X$ with the ultrametric 
$d(x,x')=e^{-\lambda_X n}$ where $n$ is the largest integer such that $x_i=x_i'$ for all $i<n$. We call it the metric with Lyapunov exponent $\lambda_X$.
The cylinder $C_m(x)$ and the open ball $B^X(x,r)$ are equal when $m=\lfloor \frac{-1}{\lambda_X}\log r\rfloor$.
Recall that the Hausdorff dimension of $\mu$ is the ratio of the entropy $h_\mu$ and the Lyapunov exponent $\lambda_X$: $d_\mu=h_\mu/\lambda_X>0$, and $d_\rho=d_\mu+d_\nu$.

\begin{theorem}\label{THMcvpslog}
Assume Hypothesis~\ref{HHH}. Then the lower recurrence rate is equal to the dimensions : 
\[
R^F(z) = \min(2d_\mu,d_\rho)\quad\text{for $\rho$-a.e. $z$}.
\]
\end{theorem}
Before proving this result, we will state some notations and useful intermediate results.
It follows from~\eqref{formuleFn} that, for a point $z=(x,y)$ we have the obvious equivalence
\begin{align}\label{Gry}
 d(F^n(z),z)<r &\text{ iff } 
 \begin{cases}
 	& d(f^n(x),x)<r\\
 	& h_n(x)\in G_r(y):=\left\{k\in\mathbb Z
 	\, :\,  d(g^ky,y)<r\right\}
\end{cases}
\, .
\end{align}

We will apply a version of the local limit theorem \cite{ps} (see \cite{yassine2} for the precised error term); quantifying the so-called mixing local limit theorem 
(See e.g.  \cite{gouezel}).
Let us write $\mathcal F_m$ for the $\sigma$-algebra of sets made of union of cylinders of generation $n$.
\begin{proposition}\label{cullt}
Assume Hypothesis~\ref{HHH}. 
There exists $C>0$ such that, for any positive integers $n,m$ satisfying $m\le 3n$, and for any $A\in\mathcal F_m$ and any $B\in\sigma\left(\bigcup_{k\ge 0}f^{-k}(\mathcal F_{k+m})\right)$,
\[
\left|\mu\left(\{x\in A, h_n(x)=k,f^n(x)\in B\}\right) - \frac{\mu(A)\mu(B)}{\sigma\sqrt{2n\pi}}
\exp\left(-\frac{k^2}{2\sigma^2n}\right)\right| \le  \mu(A)\mu(B) C\frac{m }{n}.
\]
\end{proposition}
A measurable set $B$ is in $\sigma\left(\bigcup_{k\ge 0}f^{-k}(\mathcal F_{k+m})\right)$
if there exists a measurable function $f_0:\mathcal A^{\mathbb N}\rightarrow \{0,1\}$ (where $\mathcal A$ is the alphabet of the subshift $X$) such that $ 1_B(\mathbf x)=f_0(x_m,....)$.

\begin{lemma}\label{rug}
Assume Hypothesis~\ref{HHH} and $d_\nu>0$. For any $\eps>0$, any decreasing family $(N_r)_{r>0}$ of positive integers,
for $\nu$-a.e. $y\in Y$,  for any $r>0$ sufficiently small
$$\# (G_r(y)\cap[-N_r,N_r])
\le1+2N_{r/2} r^{d_\nu
	-2\eps}, $$
where $G_r(y)$ is the set defined in~\eqref{Gry}.
\end{lemma}
\begin{proof}
Note that $k=0\in G_r(y)$. For non zero $k$, it suffices to estimate the number of positive $k$ in the set, and then to apply the same estimate to $g^{-1}$, which still satisfies our assumption, to get the result for negative $k$.
By assumption (B) on $g$ and \cite{RS}, the lower recurrence rate satisfies $\underline{R}^g(y)=d_\nu$
for $\nu$-a.e. $y\in Y$. 
Let $\eps>0$ and set
$$Y_\eps^{r_0} = \left\{y\in Y\colon
 \forall r<r_0,
 d(g^k(y),y)\ge r \text{ for } 1\le k<r^{-d_\nu+\eps} 
 \text{ and } \nu(B_{2r}^Y(y))\le r^{d_\nu-\eps} \right\}
$$
Since $\nu(Y_\eps^{r_0})\to 1$ as $r_0\to0$, it suffices to prove the results for $y\in Y_\eps^{r_0}$.
Let $y_0\in Y$ and $r<r_0$. Set $B=B_{2r}^Y(y_0)$. When $y\in B_r^Y(y_0)$ and $d(g^ky,y)<r$ we have $g^ky\in B$. Moreover, if $y\in Y_\eps^{r_0}$ this does not happen if $k< r^{-d_\nu+\eps}$. Therefore, by Markov inequality 
\begin{equation}
\nu\left( y\in B_r^Y(y_0)\cap Y_\eps^{r_0}\colon \# \{k=1..N_r\colon d(g^k(y),y)<r\}>L\right)
\le \frac1L \sum_{r^{-d_\nu+\eps}\le k\le N_r} \nu(B\cap g^{-k}B).
\end{equation}
By assumption\footnote{This follows from $\psi$-mixing when $(Y,g,\nu)$ is a SFT with an equilibrium state; otherwise it follows by approximation of indicator function of  balls by Lipschitz functions as in \cite{RS}.}, for such $k$ we have
\[
 \nu(B\cap g^{-k}B)\le 2\nu(B)^2.
\]
Taking a finite cover of $Y_\eps^{r_0}$ by balls of radius $r$ of multiplicity at most $M$ shows that
\[
\nu( y\in Y_\eps^{r_0}\colon \# G_r(y) \cap [1,N_r]>L) \le 2 M\frac {N_r}{L} r^{d_\nu-\eps}=O(r^\eps),
\]
choosing $L=N_rr^{d_\nu-2\eps}$.  The result follows from the Borel Cantelli lemma, summing up over $r_m=2^{-m}$ and then using the monotonicity of $N_r$.
\end{proof}
We follow the proof in \cite{ps}, using the extra information about the growth rate of $h_n$ given by the law of iterated logarithm.
\begin{proof}[Proof of Theorem~\ref{THMcvpslog}]

By  \eqref{majo1R} and \eqref{majo2R} we only need to prove a lower bound.

We assume that $d_\nu>0$.
Let $\varepsilon>0$ and 
$K^\eps=K_\eps^{m_0,n_0}$ be the set of points $x\in X$ such that $\forall m\ge n_0$,  $\mu(C_m(x)) \le e^{-m(h_\mu-\eps)}$ and $\forall n\ge n_0$, $|h_n(x)|\le (1+\eps)\sigma\sqrt{n\log\log n}$. The Shannon-McMillan-Breiman theorem and  the law of iterated logarithm 
ensures the existence of $m_0$ and $n_0$ such that $\mu(K_\eps^{m_0,n_0})\ge 1-\eps$. Let $N=N_n=(1+\eps)\sigma\sqrt{n\log\log n}$.
We fix now $m$ such that $m:=\lfloor -\frac{1}{\lambda_X}\log r\rfloor$.

First note that for any $y\in Y$
\begin{equation}\label{dF<r}
\mu\left(\{x\in K_\eps,d(F^n(x,y),(x,y))<r\}\right) 
\le \sum_{C_m\, :\, C_m\cap K_\eps\neq\emptyset} \mu\left(\{x\in C_m\cap f^{-n}C_m, h_n(x)\in G_r(y)\}\right).
\end{equation}
By Proposition~\ref{cullt}, for any $k\in\ZZ$, the quantity
\begin{equation}\label{bof}
\mu\left(\{x\in C_m\cap f^{-n}C_m, h_n(x)=k\}\right)
\end{equation}
is the sum of a main term 
\begin{equation}\label{mainterm}
\mu(C_m)^2\frac{1}{\sigma\sqrt{2n\pi}}\exp\left(-\frac{-k^2}{\sigma^2n}\right) \le c\mu(C_m)^2 n^{-1/2}
\end{equation}
and of an error term bounded in absolute value by
\begin{equation}
\mu(C_m)^2 \frac{m}{n}.
\end{equation}
Summing the main contribution \eqref{mainterm} among all $m$-cylinders intersecting $K^\eps$ and all integer $k\in [-N_n,N_n]$ such that $d(g^ky,y)<r$ gives, using Lemma~\ref{rug} ($r$ and $n$ will be linked later), that for $\nu$-a.e. $y$,
provided $r$ is small enough
\begin{align*}
E_n^\eps(y,r)
&:=
\sum_{k\in[-N_n,N_n]} \sum_{C_m\, :\, C_m\cap K_\eps\neq\emptyset}c\mu(C_m)^2 n^{-1/2} \\
&\le
ce^{h_\mu-\eps} r^{d_\mu-\eps/\lambda_X} n^{-1/2}\left(1+\frac{m}{\sqrt{n}}\right) \# \{k\in[-N_n,N_n]\colon  d(g^ky,y)<r\} \\
&\le 
ce^{h_\mu-\eps} r^{d_\mu-\eps/\lambda_X}\left(1+\frac{|\log r|}{\lambda_X\sqrt{n}}\right)\frac{1+r^{d_\nu-2\eps}\sqrt{n\log\log  n}}{\sqrt{n}}.
\end{align*}

If $d_\mu\le d_\nu$, we take $r_n=n^{-\frac{1}{2d_\mu-\kappa\eps}}$ with $\kappa=\max(6,3/\lambda_X)$. The term in the numerator goes to one as $n\to\infty$ therefore the whole term is bounded with
\[
E_n^\eps(y,r_n) =O( r_n^{d_\mu-\eps/\lambda_X}n^{-1/2}) = O( n^{-\frac{d_\mu-\eps/\lambda_X}{2d_\mu-\kappa\eps}-\frac12})
\]
which is summable in $n$.

If $d_\mu\ge d_\nu$ we take $r_n=n^{-\frac{1}{d_\mu+d_\nu-\kappa\eps}}$ with $\kappa=3+1/\lambda_X$. 
We get a bound
\[
E_n^\eps(y,r_n) =O\left( r_n^{d_\mu+d_\nu-(1/\lambda_X+2)\eps}\sqrt{\log\log n}\right),
\]
which is again summable in $n$.

In both cases the error term is negligible with respect to the main term, therefore by Borel Cantelli lemma we conclude that for $\rho$-a.e. $z\in Z$, $d(F^nz,z)\ge r_n$ eventually.

Letting $\eps\to0$ ends the proof of Theorem~\ref{THMcvpslog} when $d_\nu>0$.

In the case where $d_\nu=0$, by \cite{RS} we have for $\rho$-a.e. $(x,y)$
$$d_\mu=R^f(x)\le \underline{R}^F(x,y)\le \overline{R^F}(x,y)\le d_\mu$$
by \eqref{majo1R}, 
which proves the equality.
\end{proof}

\section{Convergence in distribution}\label{seccvgcedistrib}
We still consider the case where $(X,f,\mu)$ is a mixing subshift of finite type.
We will first state in Section~\ref{sec:nullmu}  a result of convergence in distribution (Proposition~\ref{propnullmu}) for the first return time  in the particular case where\footnote{Throughout this article, the notation $a_r\ll b_r$ means that $a_r=o(b_r)$, i.e. that $a_r$ is negligible with respect to $b_r$ as $r\rightarrow 0$.} $\nu(B_r^Y)\ll\mu(B_r^X)$. This first convergence result will appear as a consequence of Yassine's convergence result for $\tau_r^{\widetilde F}$ established in~\cite{yassine1}.
In a second time, in Section~\ref{sec:musinu}, we  will study the asymptotic behavior of the point process of visits to small balls.
When\footnote{The notation $a_r\ll b_r$ means that $a_r=o(b_r)$ as
$r\rightarrow 0$.}
 $\nu(B_r^Y)\ll\mu(B_r^X)$  we will retrieve a result analogous to Proposition~\ref{propnullmu}.
But we will also highlight
other behaviors when $\mu(B_r^X)\ll\nu(B_r^Y)$ or $\mu(B_r^X)\approx \nu(B_r^Y)$.
The normalization will be given by 
\[
n_r(x,y):=1/\max\left(\mu(B_r^X(x))^2,\rho(B_r^Z(x,y))\right)\, .
\]

\subsection{Study of the first return time when $\nu(B_r^Y)\ll\mu(B_r^X)$}\label{sec:nullmu}
Then we set $n_r(x,y)=(\mu(B_r^X(x)))^{-2}$.

\begin{proposition}\label{propnullmu}
Assume $(X,f,\mu)$ is a mixing two-sided subshift of finite type and that $h$ is bounded H\"older continuous, that $\nu(B_r^Y)\ll\mu(B_r^X)$ in $\rho$-probability
and that\footnote{This happens, e.g. if $(\nu(B_r^Y)\tau_r^g)_r$ and $(\nu(B_r^Y)\tau_r^{g^{-1}})_r$ both converge, as $r\rightarrow 0$, to some random variable with no atom at $0$.}
\[
\lim_{\varepsilon\rightarrow 0}\limsup_{r\rightarrow 0}\nu(\nu(B_r^Y)\tau_r^g<\varepsilon)=\lim_{\varepsilon\rightarrow 0}\limsup_{r\rightarrow 0}\nu(\nu(B_r^Y)\tau_r^{g^{-1}}<\varepsilon)=0\, .
\]
Then 
$(\mu(B_r^X(\cdot))^2\tau^F_r)_r$ converges in distribution, as $r\rightarrow 0$ to $\sigma^2\mathcal E^2/\mathcal N^2$, where $\mathcal E$ and $\mathcal N$ are standard exponential and Gaussian random variable mutually independent,
and where $\sigma^2$ is the asymptotic variance of $(h_n/\sqrt{n})_n$.
\end{proposition}
\begin{proof}
Let $\varepsilon>0$ and $\eta>0$.
Set $D_{r,\eta}:=\{\nu(B_r^Y(\cdot))\le \eta\mu(B_r^X(\cdot))\}$.
Furthermore, 
\begin{align}
\rho&\left(\sup_{k\le n_r}|h_k|> \varepsilon/\nu(B_r^Y(\cdot)),\ D_{r,\eta}\right)
\le   \rho\left(\sup_{k\le \eta^2(\nu(B_r^Y(\cdot)))^{-2}}|h_k|> \varepsilon/\nu(B_r^Y(\cdot))\right)\nonumber
\\
&\le \int_Y\mu\left(\sup_{k\le \eta^2(\nu(B_r^Y(y)))^{-2}}|h_k|> \varepsilon/\nu(B_r^Y(y))\right)\, d\nu(y)\, ,\label{rhosuphk}
\end{align}
which converges to $\mathbb P(\eta W>\varepsilon)$ as $r\rightarrow 0$  since $\sup_{k\le \eta^2n}\frac{|h_k|}{\sqrt{n}}$ converges to $\eta W$ (where $W=\sigma \sup_{[0;1]}|B|$, $B$ being a standard Brownian motion).
Let 
\[
\Omega_{r,\varepsilon}:=\left\{\min(\tau_r^g,\tau_r^{g^{-1}})> \varepsilon/\nu(B_r^Y(\cdot)),\ \sup_{k\le n_r}|h_k|\le\varepsilon/\nu(B_r^Y(\cdot))\right\}\, .
\]
On $\Omega_{r,\varepsilon}$, for 
all $n=1,...,n_r$,
\[
[d(f^{n}(x),x)<r,\ d(g^{h_n(x)}(y),y)<r]\quad\Leftrightarrow\quad
[h_n(x)=0, \ d(f^{n}(x),x)<r]\, .
\]
Thus, on $\Omega_{r,\varepsilon}$,
\[
\tau_r^F(x,y)=\tau^{\widetilde F}_r(x):=\inf\left\{n\ge 1\, :\, h_n(x)=0, \ d(f^{n}(x),x)<r\right\}\, .
\]
But Yassine proved in \cite{yassine1} that, when $(X,f,\mu)$ is a subshift of finite type, then $(\mu(B_r^X(\cdot))^2\widetilde\tau_r)_r$ converges in distribution, with respect to $\mu$ (and so to $\rho$) as $r\rightarrow 0$, to $\mathcal E^2/\mathcal N^2$.
We conclude as follows. For all $t>0$,
\begin{align*}
\limsup_{r\rightarrow 0}&\left|\rho((\mu(B_r^X(\cdot))^2\tau^F_r>t)-\mathbb P(\mathcal E^2/\mathcal N^2>t)\right|\\
&\quad\quad\le \limsup_{r\rightarrow 0}\left[\rho(\Omega_{r,\varepsilon}^c)+\left|\rho(\Omega_{r,\varepsilon},\ (\mu(B_r^X(\cdot))^2\tau^F_r>t)-\mathbb P(\mathcal E^2/\mathcal N^2>t)\right|\right]|\\
&\quad\quad\le \limsup_{r\rightarrow 0}\left[2\rho(\Omega_{r,\varepsilon}^c)+\left|\rho((\mu(B_r^X(\cdot))^2\tau^{\widetilde F}_r>t)-\mathbb P(\mathcal E^2/\mathcal N^2>t)\right|
\right]\\
&\quad\quad\le 2 \limsup_{r\rightarrow 0}\rho(\Omega_{r,\varepsilon}^c)\, .
\end{align*}
Moreover, it follows from the convergence of~\eqref{rhosuphk} that, for all $\varepsilon,\eta$,
\begin{align*}
\limsup_{r\rightarrow 0}\rho(\Omega_{r,\varepsilon}^c)
&\le
\limsup_{r\rightarrow 0}[\nu(\tau_r^g\le \varepsilon/\nu(B_r^Y))+\nu(\tau_r^{g^{-1}}\le \varepsilon/\nu(B_r^Y))+\rho(D_{r,\eta}^c)]+\mathbb P(W>\varepsilon/\eta)\\
&\le \limsup_{r\rightarrow 0}[\nu(\tau_r^g\le \varepsilon/\nu(B_r^Y))+\nu(\tau_r^{g^{-1}}\le \varepsilon/\nu(B_r^Y))]+\mathbb P(W>\varepsilon/\eta)\\
\end{align*}
We end the proof of Proposition~\ref{propnullmu} by taking $\limsup_{\varepsilon\rightarrow 0}\limsup_{\eta\rightarrow 0}$.
\end{proof}

\subsection{Study of the point process in the general case}\label{sec:musinu}
We are interested in the study of the asymptotic behavior of
the point process generated by visits to the ball $B_r^Z(x,y)=B_r^X(x)\times B_r^Y(y)$, i.e.
\begin{equation}\label{formulaNr}
\N_r (z)=\sum_{n\in\NN\, :\, F^n(z)\in B_r^Z(z)}  \delta_{n/n_r}\, .
\end{equation}
To this end, we will consider moments of the multivariate variable $(\N_r([t_{v-1};t_v]))_v$. 
To simplify the exposure of our proofs, we have chosen to restrict our study to the following case.
\begin{hypothesis}\label{hypsubshift}
We assume that
\begin{itemize}
\item[(I)] The system $(X,f,\mu)$ is a one-sided mixing subshift of finite type and $\mu$ is an equilibrium state of a (normalized) Hölder potential\footnote{In particular
the ball $B_r^X(x)$ 
corresponds to the $|\log r|$-cylinder containing $x$, i.e. to the set of points $(y_k)_{k\in\mathbb Z}$ such that $y_k=x_k$ for all non-negative integer $k\le |\frac{1}{\lambda_X}\log r|$.}
\item[(II)] The system $(Y,g,\nu)$ is a two-sided mixing subshift of finite type and the measure $\nu$ is an equilibrium state of a Hölder potential, or, more generally, it  satisfies the following condition: For all integers $J,K$ such that
$2\le J\le K$, there exists 
$\alpha\in(0;1)$ and $c_0\ge 1$ such that, for all integers $\ell_1<...<\ell_K$ and all $y\in Y$, the following holds true\footnote{Note that for mixing subshifts of finite type this assumption holds also true if we replace the 2-sided cylinders $B_r^Y(y)=\{z\, :\, z_k=y_k,\ \forall |k|\le|\frac{1}{\lambda_Y}\log r|\}$ by the one-sided cylinders $\{z\, :\, z_k=y_k,\ \forall k=0,...,\lfloor|\frac{1}{\lambda_Y}\log r|\rfloor\}$.}
\[
\nu\left(\bigcap_{j=1}^Kg^{-\ell_j}(B_r^Y(y))\right)=\left(1+\mathcal O(\alpha^{\ell_{J}-\ell_{J-1}-c_0\log r})\right)\nu\left(\bigcap_{j=1}^{J-1}g^{-\ell_j}(B_r^Y(y))\right)\nu\left(\bigcap_{j=J}^{K}g^{-\ell_j}(B_r^Y(y))\right)\, ,
\]
uniformly in $(r,y,\ell_1,...,\ell_K)$.
\item[(III)] The $\mu$-centered function $h$ is constant on $0$-cylinders
 with asymptotic variance $\sigma^2:=\lim_{n\rightarrow +\infty}\frac{\mathbb E_\mu[h_n^2]}n$,
\item[(IV)] The function $h$ is non-arithmetic, i.e. $h$ is not cohomologous in $L^2(\mu)$ to a sublattice valued function.
 \end{itemize}
\end{hypothesis}
Under these assumptions, we know that $(h_{\lfloor nt\rfloor}/\sqrt{n})_{n\ge 1}$ converges in distribution, as $n\rightarrow +\infty$, to a centered Brownian process $B$ of variance $\sigma^2$.
Let 
$(L_t(x))_{t\ge 0,x\in\mathbb R}$ be a continuous compactly supported version of the local time of $B$, i.e. $(L_t(x))_{t,x}$ satisfies
\[
\int_{\mathbb R}f(x) L_t(x)\, dx=\int_0^t f(B_s)\, ds\, ,
\]
i.e. $L_t$ is the image measure of the Lebesgue measure on $[0;t]$ by the Brownian motion $B$. 
Recall that
\[
n_r=n_r(x,y)=\min\left((\mu(B_r^X(x)))^{-2},(\mu(B_r^X(x))\nu(B_r^Y(y)))^{-1}\right)\, .
\]
We define 
$\alpha_r(x,y):= n_r(x,y)\mu(B_r^X(x))^2$ and $\beta_r(x,y):=n_r(x,y)\rho(B_r^Z(x,y))$. Note that 
\begin{equation}\label{arbr}
(\alpha_r(x,y),\beta_r(x,y)) = 
\begin{cases} 
\left(1,\frac{\nu(B_r^Y(y))}{\mu(B_r^X(x))}\right) &\text{ if } \mu(B_r^X(x))>\nu(B_r^Y(y))\\
\left(\frac{\mu(B_r^X(x))}{\nu(B_r^Y(y))},1\right) &\text{ otherwise}
\end{cases}
\, .
\end{equation}
Let us now state our key result that will be proved in Section~\ref{secappmom}.
\begin{theorem}\label{THMcvloi}
Assume Hypothesis~\ref{hypsubshift}. 
Let $K$ be a positive integer and $\mathbf{\overline m}=(\overline m_1,...,\overline m_K)$ be a $K$-uple of positive integers and let $(t_0=0,t_1,...,t_K)$ be an increasing collection of nonnegative real numbers.
There exist $C>0$ and $u>0$ such that, for every $(x,y)\in X\times Y$,
\begin{align}
&\left|\mathbb E_\rho\left[\left.\prod_{v=1}^K \mathcal N_r(]t_{v-1},t_{v}])^{\overline m_v}\right|B_r^Z(x,y)\right]- \mathbb E
\left[\prod_{v=1}^K\left(\mathcal Z^{(x,y)}_r(t_v)-\mathcal Z^{(x,y)}_r(t_{v-1})\right)^{\overline m_v}\right]\right|\nonumber\\
&\le C\left( |\log r|^{3^{m}+m} \left(\mu\left(\left.
\tau^f_{B_r^X(x)}\le |\log r|^{3^m}\right|B_r^X(x)\right)+\nu\left(\left.
\tau^g_{B_r^Y(y)}\le |\log r|^{3^m}\right|B_r^Y(y)\right)\right.\right.\label{error1}\\
&\left. \left.
+e^{-u\sqrt{-\log r}}+r^{\frac{m^2}2}+\frac{\log n_r}{\sqrt{n_r}}\right)
+|\log r|^{-\frac 12}
+ \varepsilon_0(|\log r|^2/n_r)
\right)
\, ,
\label{error2}
\end{align}
with $m=|\mathbf{\overline m}|=\overline m_1+\cdots+\overline m_K$, $\varepsilon_0$ bounded, continuous, vanishing at 0, 
and with 
\[
\mathcal Z_r^{(x,y)}=\mathcal Z_{\alpha_r(x,y),\beta_r(x,y)}
\, ,
\]
where $\mathcal Z_{0,1}$ is a standard Poisson process and where, for all $\alpha\in(0,1]$ and all $\beta\in[0;1]$, \begin{equation}\label{Zab}
\mathcal Z_{\alpha,\beta}(t)=\int_{\mathbb R}\mathcal P'_s(L_{t}(s))\, d(\delta_0+\mathcal P)(s)\, ,
\end{equation}
where  $\mathcal P$, $\mathcal B$ and $(\mathcal P'_s)$
are mutually independent, $\mathcal B$ being a Brownian motion of variance $\sigma^2$ and of local time $L$, $(\mathcal P'_s)_{s\in \mathbb R}$ being a family of independent homogeneous Poisson processes with intensity 
$\sqrt{\alpha}$ and $\mathcal P$ being a two-sided Poisson process with intensity 
$\beta/\sqrt{\alpha}$.
\end{theorem}
\begin{remark}\label{Z10}
Observe that $\mathcal Z_{1,0}(t)=\mathcal P'_0(L_t(0))$. 
Furthermore, we will see in Appendix~\ref{append} that
the moments 
$\mathbb E
\left[\prod_{v=1}^K\left(\mathcal Z_{\alpha,\beta}(t_v)-\mathcal Z_{\alpha,\beta}(t_{v-1})\right)^{m'_v}\right]$
 are continuous in $(\alpha,\beta)$.
\end{remark}
\begin{corollary}[Conditional convergence in distribution]\label{corCVloi}
Assume the assumptions and keep the notations of Theorem~\ref{THMcvloi}. 
Suppose that $(x,y)\in X\times Y$ is such that the limit
$$\lim_{r\to0}(\alpha_r(x,y),\beta_r(x,y))=:(\alpha,\beta)$$ 
exist
and all the error terms of Theorem~\ref{THMcvloi} satisfy $\eqref{error1}+\eqref{error2}=o(1)$ as $r\rightarrow 0$.
Then $\mathcal N_r$ converges in distribution for the vague convergence\footnote{See e.g.~\cite{Resnick} for this convergence. 
Recall that this convergence implies also the convergence in distribution of $(\mathcal N_r([0;t]))_{t\in[0;T]}$ (seen as a c\`adl\`ag process) for the $J1$-metric (see \cite{Jagers}).} as $r\rightarrow 0$, with respect to $\rho(\cdot|B_r^X(x)\times B_r^Y(y))$, 
to $\mathcal Z_{\alpha,\beta}$ defined in~\eqref{Zab}. 

\end{corollary}
\begin{proof}
It follows from Theorem~\ref{THMcvloi} that the moments of any linear combination of the coordinates of multivariate variables
\begin{equation}\label{multi}
X_r:=(\mathcal N_r(]t_{v-1},t_{v}]))_{v=1..K}
\end{equation} converges to those of $X_0:=(\mathcal Z_{\alpha,\beta}]t_{v-1},t_v])_{v=1..K}$.

Note that $\|X_0\|_\infty\le \mathcal Z_{\alpha,\beta}([t_0,t_K])=:Y$, which is a random variable with Poisson distribution
with random parameter bounded by $c(1+|N|)$ where $N$ is a standard normal random variable.
The convergence in distribution of $X_r$ (an thus of the process $\mathcal N_r$ itself by convergence of its finite dimensional distributions) will follow from the multivariate Carleman's criterion (See Lemma~\ref{carl+}) provided
\[
\sum_{m\ge 1}\mathbb E[Y^m]^{-\frac 1{m}}=\infty\, .
\]
It follows from Lemma~\ref{mompoi} (with the notations therein) that, for $m\ge1$,
\begin{align*}
\mathbb E[Y^m]&=\sum_{q=1}^m 
S(m,q) c^q\sum_{k=0}^q{n\choose k}\frac{\Gamma(1+\frac 12)^k}{\Gamma(1+\frac{k}2)}\\
&\le (\sqrt{\pi}/2)^m\sum_{q=1}^m S(m,q) (2c)^q\\
&\le (2(1+c)\sqrt{\pi}/2)^mm^m \, ,
\end{align*}
since $\Gamma(1+\frac 12)=\sqrt{\pi}/2$, $\Gamma(1+\frac k2)\ge 1$
and $\sum_{k=0}^q{q\choose k}=2^q$, and finally since
the number of partitions of $\{1,...,m\}$ in non-empty sets is dominated by the number of its self maps $m^m$.
Thus
\[
\sum_{m\ge 1}\mathbb E[Y^m]^{-\frac 1{m}}\ge \sum_{m\ge 1}(2m(1+c)\sqrt{\pi}/2))^{-1}=+\infty \, .
\]
This implies the convergence of the finite distributions, which, combined with the convergence of their moments, implies the convergence in distribution in the space of positive measure endowed with the vague convergence (due to \cite[Proposition 3.22]{Resnick}).
\end{proof}

\begin{theorem}\label{thmcvloi2}
Assume Hypothesis~\ref{hypsubshift}, that both systems are SFT
(either with 2-sided or 1-sided cylinders), and that the error terms $\eqref{error1}+\eqref{error2}$ of Theorem~\ref{THMcvloi} converge to $0$ in $\rho$-probability.

Assume furthermore that $(\alpha_r,\beta_r)$ converges in distribution, under $\rho$, to some random variable with law $\eta$.

Then $\mathcal N_r$
converges in distribution, with respect to $\rho$ for the vague convergence, as $r\rightarrow 0$, to the point process $\mathcal Z_\pi$ where $\pi$ is a random variable independent of the $\mathcal Z_{\alpha,\beta}$'s and with distribution $\eta$. Namely
\[
\mathbb E[\phi(\mathcal Z_\pi)]=
\int_{[0,1]^2}\mathbb E\left[\phi(\mathcal Z_{\alpha,\beta})\right]\, d\eta(\alpha,\beta)\, ,
\]
for all bounded continuous function $\phi$ defined on the space of 
measures on $(0,+\infty)$.

In particular, if $\pi$ is a.s. constant equal to $(\alpha,\beta)$ 
then $\mathcal N_r$ converges in distribution, as $r\rightarrow 0$ and with respect to $\rho$, to $\mathcal Z_{\alpha,\beta}$ defined in~\eqref{Zab}. 
\end{theorem}
\begin{proof}
Fix some $K$ and $(t_v)$'s as in Theorem~\ref{THMcvloi}.
We will apply the multivariate moments Lemma~\ref{carl+}  to prove the convergence in distribution 
$$
X_r=(\mathcal N_r([t_{v-1},t_{v}]))_{v=1..K} \cvto{r\to0} X_0:=(\mathcal Z_{\pi}([t_{v-1},t_v]))_{v=1..K}.
$$

The Carleman's criterion holds as in the proof of Corollary~\ref{corCVloi}. 

For any multi-index $m'=(m_1',\ldots m_K')$ and $r>0$ denote the corresponding error term in Theorem~\ref{THMcvloi} by $\epsilon_r^{(m')}(z):=\eqref{error1}+\eqref{error2}$. 
By assumption $\epsilon_r^{(m')}\to0$ in probability as $r\to0$.
Set $\Delta_r^{(d)}:=\max_{|m'|\le d}\epsilon_r^{(m')}$. We still have $\Delta_r^{(d)}\to0$ in probability as $r\to0$.
Therefore, for all $d\ge 1$, there exists $r_d\in(0;1]$ (take the largest one) such that 
$$
\forall r\in]0;r_d[,\quad \rho(\Delta_r^{(d)}>d^{-1})<d^{-1}.
$$
The sequence $(r_d)_d$ is decreasing.\\
If it converges to some value $r_\infty>0$,
this means that, for all $r\in]0;r_\infty[$ and all $d$, $\rho(\Delta_r^{(d)}>d^{-1})<d^{-1}$, and so that $\rho(\Delta_r^{(\infty)}>d^{-1})<d^{-1}$, so that $\Delta_r^{(\infty)}=0$ a.s. and we can take $\Omega_r=\{\Delta_r^{(\infty)}=0\}$.
If $(r_d)_d$ converges to $0$, then, for any $d\ge 1$ and any $r\in]r_{d+1}; r_d]$, we set
\[
\Omega_r:=\left\{\Delta_r^{(d)}\le \frac 1d\right\}\, .
\]
We notice that $\rho(\Omega_r)\to1$ as $r\to0$.

Let $\mathbf {m'}$ and set $m=|\mathbf{m'}|$ as in Theorem~\ref{THMcvloi}.  For $(\alpha,\beta)\in[0,1]^2$, set $G(\alpha,\beta):=\mathbb E[\prod_{v=1}^K \mathcal Z_{\alpha,\beta}(]t_{v-1},t_v])^{m_v'}]$.

We partition the space $X\times Y$ by balls $B_r^Z$ of radius $r$, noticing that $\Omega_r$ is a finite union of such balls and that $\alpha_r$ and $\beta_r$ are constants on these balls and get
\begin{align*}
E_\rho(\prod_{v=1}^K(\mathcal N_r(]t_{v-1};t_v]))^{m'_v}1_{\Omega_r^{d}})
&=\sum_{B_r^Z\colon B_r^Z\subset \Omega_r }
\rho(B_r^Z)E_\rho\left[\left.\prod_{v=1}^K(\mathcal N_r(]t_{v-1};t_v]))^{m'_v}\right|B_r^Z\right] \\
&= \int_{\Omega_r} G(\alpha_r(B_r^Z),\beta_r(B_r^Z)) d\rho + O(\sup_{\Omega_r}\Delta_r^{(m)})
\\
&\cvto{r\to0}\int_{[0,1]^2}G(\alpha,\beta) d\eta(\alpha,\beta)
=\mathbb E\left[\prod_{v=1}^K(\mathcal \mathcal Z_\pi(]t_{v-1};t_v]))^{m'_v}\right]\, ,
\end{align*}
by the convergence in distribution of $(\alpha_r,\beta_r)$ and the continuity and boundedness of $F$ (coming from Lemmas~\ref{mompoi} and~\ref{momentPoissonInt}).
\end{proof}
\begin{corollary}\label{corosubshift} Assume Hypothesis~\ref{hypsubshift}, that both systems are SFT ($f$ one-sided, $g$ 2-sided) with equilibrium states of Hölder potentials. Let $\lambda_Y>0$ and endow $Y$ with the metric (resp. pseudo metric) with Lyapunov exponent $\lambda_Y$, so that $B_r^Y$ are two-sided (resp. one-sided) cylinders; Set $d=2$ (resp. $d=1$). 
Then in the following cases, $\mathcal N_r$ converges in distribution to the random process $\mathcal Z_\pi$, where the random parameter $\pi$ is equal to 
\begin{itemize}
	\item[(a)] $\pi=(1,0)$ a.s. if  $0<d_\mu<d_\nu$;
	\item[(b)] $\pi=(0,1)$ a.s. if  $d_\mu>d_\nu>0$;
	\item[(c)] $\pi=(1,0)$ or $(0,1)$ with probability $1/2$ if $d_\mu=d_\nu>0$ and if at least one of the measures is not of maximal entropy;
	\item[(d)] $\pi$ is a discrete random variable supported on $(0,1]^2$ if $d_\mu=d_\nu$, $\lambda_X=\lambda_Y$ and both measures are of maximal entropy\footnote{The two measures are thus Markov. The limit distribution $\pi$ is explicitly computed in terms of the stationary vector at the end of the proof of Corollary~\ref{corosubshift}.}.
	In particular, if $f$ and $g$ are two full shifts with uniform distribution (on sets of respectively $L^d$ and $L$ elements), then $\pi=(L^{1-d},1)$.
\end{itemize}
\end{corollary}
The remaining case (d) with $\lambda_X\neq\lambda_Y$ will be considered in Remark~\ref{dege2} below.
\begin{proof}
To apply Theorem~\ref{thmcvloi2} we first show that the error terms
\eqref{error1} and \eqref{error2}
 go to zero.
Let $0<\gamma<\dim_H\mu=h_\mu$.  
Given $r>0$, a summation over balls $B_r^X$ allows to get
\[\sum_{B_r^X}\mu(B_r^X)\mu\left(\left.\tau^f_{B_r^X}\le |\log r|^{3^m}\right|B_r^X\right) \le \mu(\tau_r^f\le r^{-\gamma}) = o(r^{\alpha}),\]
by the large deviation estimates for return time proven in~\cite{crs}, for some $\alpha>0$ and all $r$ sufficiently small. 
The same arguments apply to the error term involving $\tau_r^g$. Therefore the error terms go to zero in $L_\rho^1$, hence in probability.

By the existence of the pointwise dimension we have $\rho-$ a.e. $(x,y)$
$$
\frac{1}{\log r}\log{\frac{\mu(B_r(x))}{\nu(B_r(y))}}\cvto{r\to0}  d_\mu- d_\nu.
$$
Using \eqref{arbr} we get in the first case (a) that  $(\alpha_r,\beta_r)\cvto{r\to0}(1,0)$ a.s., and in the second case (b), $(\alpha_r,\beta_r)\cvto{r\to0}(0,1)$ a.s.

Suppose that $d_\mu=d_\nu$ and that at least one of the measures is not of maximal entropy. Let $k_r^X=\lfloor\frac{-1}{\lambda_X}\log r\rfloor$ and $k_r^Y=\lfloor\frac{-1}{\lambda_Y}\log r\rfloor$. Then for $\phi_\mu$ and $\phi_\nu$ the respective normalized potentials we have
$$
\frac{1}{\sqrt{|\log r|}}\log{\frac{\mu(B_r^X(x))}{\nu(B_r^Y(y))}} = \frac{1}{\sqrt{k_r^X}}\sum_{j=0}^{k_r^X}[\phi_\mu(f^jx)+h_\mu]-\frac{1}{\sqrt{k_r^Y}}\sum_{j=(1-d)k_r^Y}^{k_r^Y-1}[\phi_\nu(g^ky)+h_\nu]+o(1),
$$
since $k_r^Xh_\mu-dk_r^Yh_\nu = O(1)$.
By the central limit theorem (for $f$ and $g$) these normalized Birkhoff sums converge in distribution under the product measure $\rho$ to a sum $S$ of two centered (since $\int \phi_\mu d\mu =-h_\mu$ and $\int\phi_\nu d\nu=-h_\nu$) independent gaussian random variables, with at least one of them of nonzero variance (otherwise both potentials are cohomologous to a constant and each measure has maximal entropy).
Hence the variance of $S$ is positive. Therefore, removing the normalization, the ratio of the measures converges in distribution to the uniform law on $\{-\infty,+\infty\}$, proving the result in the third case (c).

In the last case (d), the two measures are the Parry measure. Let $A^X$ and $A^Y$ denote the transition matrices of the subshifts. Denote by $u^X,v^X$ and $u^Y,v^Y$ their left and right positive eigenvectors, associated to the maximal eigenvalues $e^{h_\mu},e^{h_\nu}$. We fix the normalization $u^X\cdot v^X=1$ and $u^Y\cdot v^Y=1$. 
Dropping the dependence on $r$ we denote by $k$ the common value of $k_r^X=k_r^Y$ (since $\lambda_X=\lambda_Y$).
The measure of a cylinder is known to be equal to $\mu([a_0 a_1\ldots a_k])=u^X_{a_0}v^X_{a_k}e^{-kh_\mu}$ and similarly for $\nu$.
Thus the distribution of the ratios 
\begin{equation}\label{ratio}
\frac{\mu([a_0\ldots a_k])}{\nu([b_{-k}\ldots b_k])} = \frac{u^X_{a_0}v^X_{a_k}}{u^Y_{b_{-k}}v^Y_{b_k}}\mbox{ if }d=2\quad\mbox{and}\quad
\frac{\mu([a_0\ldots a_k])}{\nu([b_{0}\ldots b_k])} = \frac{u^X_{a_0}v^X_{a_k}}{u^Y_{b_{0}}v^Y_{b_k}}\mbox{ if }d=1
\end{equation}
is given by the discrete measure
$$
\sum_{a,a',b,b'}\mu([a]\cap f^{-k}[a'])\nu([b]\cap g^{-dk}[b'])\delta_\frac{u^X_a v^X_{a'}}{u^Y_b v^Y_{b'}}
\cvto{k\to\infty} 
\sum_{a,a',b,b'}u^X_av^X_au^X_{a'}v^X_{a'}u^Y_bv^Y_bu^Y_{b'}v^Y_{b'}
\delta_\frac{u^X_a v^X_{a'}}{u^Y_b v^Y_{b'}}=:\eta^*,
$$
by mixing.
The limiting distribution $\eta$ is obtained by the continuous mapping theorem, applying the map $\Phi\colon \lambda\mapsto(1,\lambda^{-1})1_{\lambda>1}+(\lambda,1)1_{\lambda\le 1}$ to the distribution $\eta^*$.

In the particular case of two full shifts,
$\frac{\mu([a_0\ldots a_k])}{\nu([b_{-k}\ldots b_k])} = \frac{L^{-2(k+1)}}{L^{-(2k+1)}}=L^{-1}
$ if $d=2$ and $\frac{\mu([a_0\ldots a_k])}{\nu([b_{0}\ldots b_k])} = \frac{L^{-k-1}}{L^{-k-1}}=1
$ if $d=1$.
\end{proof}
\begin{remark}\label{dege2}
In the case when $d_\mu=d_\nu$ and both measures are of maximal entropy but $\lambda_X\neq\lambda_Y$, the random parameter does not converge. Indeed, the computations in the proof of (d) can be rewritten with $k_X$ and $k_Y$ defined as in the proof of (b). However, the entropic part in the expression of the measure of the cylinders do not cancel completely, so that in front of the ratio of the measures \eqref{ratio} a deterministic prefactor
\[
\zeta_r:=\exp\left[-\left(\lfloor\frac{-\log r}{\lambda_X}\rfloor\lambda_X-\lfloor\frac{-\log r}{\lambda_Y}\rfloor\lambda_Y\right)\frac{h_\mu}{h_\nu}\right]
\]
subsists. 
Note that if $\lambda_X$ and $\lambda_Y$ are rationally free, the accumulation points of $\zeta_r$ as $r\to0$ is the whole interval $[e^{-dh_\nu},e^{h_\mu}]$.
Proceeding as in the proof of case (d), we conclude that $\mathcal{N}_r$ is asymptotic to $\Z_{\pi_r}$, where the parameter $\pi_r$ is distributed as the image of $\eta^*$ by the continuous map $\Phi(\zeta_r\cdot)$.
\end{remark}

\begin{remark}
Under the assumptions of the previous corollary,
if $\nu(B_r^Y)\ll\mu(B_r^X)$ in $\rho$-probability, we retrieve the conclusion of Proposition~\ref{propnullmu}. 
Indeed
\[
\tau^F_r=\inf\{t>0\, :\, \mathcal P(L_t(0))\ge 1\}=\inf\{t>0\, :\, L_t(0)\ge \mathcal E\}=T^{(0)}_{\mathcal E}
\]
where 
$\mathcal E:=\inf\{u>0\, :\, \mathcal P(u)\ge 1\}$ has exponential distribution of parameter $1$ and
$T^{(0)}_u:=\inf\{t>0\, :\, L_t(0)\ge u\}$ which has the same distribution as $\sigma^2u^2\mathcal N^{-2}$
where $\mathcal N$ is a standard gaussian random variable 
(see e.g.~\cite{PitmanYor})
 combined with the fact that $L_t(0)=L'_t(0)/\sigma$ where $L'$ is the local time of the standard Brownian motion $B/\sigma$.
\end{remark}
We end this section by stating a result that ensures 
that the limit process of $\mathcal N_r$
when $\nu(B_r^Y)\ll\mu(B_r^X)$ coincide with the limit of the analogous
time process $\widetilde {\mathcal N}_r$  of return times of the $\mathbb Z$-extension $\widetilde F$ to the origin. This point process is given by
\begin{equation}\label{tildeN}
\forall x\in X,\quad \widetilde \N_r (x)=\sum_{n\in\NN\, :\, \widetilde F^n(x,0)\in B_r^X(x)\times\{0\}}  \delta_{n \mu(B_r^X(x))^2}\, ,
\end{equation}
where $\widetilde F$ has been defined in \eqref{eq:Zextension}.
We will prove in Section~\ref{secZextension} the next result
about the asymptotic behaviour of $\widetilde{\mathcal N}_r$ as $r\rightarrow 0$.
\begin{theorem}\label{ThmZextension}
	The family of point processes $(\widetilde{\mathcal N}_r)_{r>0}$ converges in distribution to $\mathcal Z_{1,0}$, as $r\rightarrow 0$, for the vague convergence, with respect to
	both 
	$\mu(\cdot|B_r^X(x))$ and $\mu$.
\end{theorem}

\section{Approximation of moments of the hitting process}\label{secappmom}
We prove here Theorem~\ref{THMcvloi}. 
Let $P$ be the transfer
operator of $(X,f,\mu)$, i.e.
\[
\forall G,H\in L^2(\mu),\quad \int_X P(G).H\, d\mu =\int_X G.H\circ f\, d\mu\, .
\]
The following results come from Fourier perturbations $u\in\mathbb C, w\mapsto P_u(w):=P(e^{iuh}w)$
 of the transfer operator $P$ acting on the Banach space $\mathcal{B_\theta}$ of $\theta$-Hölder continuous functions endowed with the norm $\|w\| := |w|_\theta + \|w\|_1$ where $|w|_\theta:=\inf\{ K>0\colon \forall n,\, \var_n(w)\le K\theta^n\}$, where $\var_n(w)$ is the maximal variation of $w$ on a $n$-cylinder, that is
 \[
 \var_n(w):=\sup_{x,y\in X:x_0=y_0,...,x_n=y_n}|w(x)-w(y)|\, .
 \]
 We recall that $P_u^n(w)=P^n(e^{iuh_n}w)$, using again the notation $h_n:=\sum_{k=0}^{n-1}h\circ f^k$.
 
 
\begin{proposition}\label{pro:Pu}
	Assume Hypothesis~\ref{hypsubshift}. 
There exist three positive numbers $\delta$, $c'$ and $\alpha<1$ and three continuous functions $u\mapsto\lambda_u\in\mathbb C$, $u\mapsto\Pi_u\in\mathcal L(\mathcal B)$ and $u\mapsto N_u\in\mathcal L(\mathcal B)$ defined on $\{z\in\mathbb C\, :\, |z|\le \delta\}$ such that 
\begin{enumerate}
	\item
	for $|u|<\delta$,  $P_u^kw = \lambda_u^k \Pi_u(w) + N_u^k(w)$ with $\|N_u^k(w)\|\le c'\alpha^k \|w\|$, 
	
	$\lambda_u = 1-\frac{\sigma^2} 2 u^2+O(u^3)=e^{-\frac{\sigma^2} 2 u^2+O(u^3)}$, $\left|\log(\lambda_u)+\frac{\sigma^2} 2 u^2\right| \le 
	\frac{\sigma^2}4 |u|^2
$ and $\|\Pi_u(w)-\mu(w)\|\le c'|u| \|w\|$,
	\item
	for $u\in|-\pi,\pi]\setminus[-\delta,\delta]$, $\|P_u^kw\|\le c'\alpha^k \|w\|$\, .
\end{enumerate}
\end{proposition}


\begin{lemma}\label{PMH}
	For all constant $c>0$ there exists a constant $C>0$ such that for any $M\ge 1$
	and any function $H$ such that $\log |H|$ is uniformly $\theta$-H\"older continuous on each $M$-cylinder with H\"older constant bounded by $c\theta^{-M}$, 
	\[
	i.e.\quad \forall D\in\mathcal C_M,\quad\forall y,z\in D,\quad 
	|H(y)|\le |H(z)|e^{c \theta^{-M}d_\theta(y,z)}\, ,
	\]
	then	for all $u\in[-\pi,\pi]$, $\|P_u^M(
		 H)\| \le C \| 
		  H \|_1$\, .
\end{lemma}
\begin{proof}
Let us write $\varphi$ for the (normalized) potential.
	Write $H=\sum_D H
	 1_D$ where the sum runs over all $M$-cylinders.
	One has $\var_n(P_u^M(
	H
	 1_D))\le\left(|\varphi+iuh|_\theta \frac{\theta^n}{1-\theta}+e^c c
	  \theta^{n}\right) \|P^M(|H
	  | 1_D)\|_\infty $.
	Note that $|\varphi +iuh|_\theta \le |\varphi|_\theta+\pi|h|_\theta$ is uniformly bounded, and \[\left|P^M(
	|H
	| 1_D)(x)\right|=|
	H
	(x_D)|\exp (S_M\varphi(x_D))\le 
	|
	H
	(x_D)| \kappa\, 
	\mu(D)\] by the Gibbs property, where $x_D \in D$ is the unique preimage in $D$ of $x\in X$ by $\sigma^M$, for some constant $\kappa$. 
	Furthermore
	\[
	\int_D |H
	(x_D)|\, d\mu(y)
	 \le \int_D |H
	 (y)|\, |H
	 (x_D)/H
	 (y)|\, d\mu(y)
	 \le e^{c}\int_D |H
	 (y)|\, d\mu(y) \, .
	\]
	Hence $|P_u^M(H
	 1_D)|_\theta \le C' \Vert H
	  1_D\Vert_1$ for the constant $C'=\left(\frac{|\varphi|_\theta+\pi|h|_\theta}{1-\theta}+c e^c\right)\kappa e^c$. Therefore, summing over $D$ gives $|P_u^M(H )|_\theta\le C'\|H\|_1$.
\end{proof}
We will use the next lemma which is the operator estimate that is behind Proposition~\ref{cullt}.
\begin{lemma}\label{PuKH} For all $c>0$, there exists a constant $C>0$ such that, for any positive integer $M$, any function $H$ as in Lemma~\ref{PMH} and any  $k>M^2$ we have
	\begin{equation}
	\sup_{\ell\in\mathbb Z}\left\| P^k\left(1_{\{h_k=\ell\}} H\right)-\frac{1}{\sqrt{k}}\Phi\left(\frac{\ell}{\sqrt{k}}\right)\mu( H)\right\| \le \frac{C}{k} \| H\|_1\, ,
	\end{equation}
where $\Phi$ is the density function of the centered Gaussian distribution with variance $\sigma^2$.

\end{lemma}	
\begin{proof}
We start with the identity 
\begin{equation}\label{eq1}
P^k\left(1_{\{h_k=\ell\}}
 H\right)=\frac{1}{2\pi}\int_{[-\pi,\pi]}e^{-i u \ell} P_u^k( H) du\, .
\end{equation}

Let $w=P_u^M( H)$. Let $d_u(k,w)=\left| P_u^{k-M}w-e^{-\frac{\sigma^2}2 u^2(k-M)}\mu(w)\right|$.
We apply Proposition~\ref{pro:Pu}.
For $|u|\ge\delta$ we have $$d_u(k,w)\le \left(c'\alpha^{k-M}+e^{-\frac{\sigma^2}2\delta^2(k-M)}\right)\|w\|.$$
For $|u|<\delta$ we have 
\begin{align*}
d_u(k,w)
&\le |P_u^{k-M}w-\lambda_u^{k-M}\Pi_u(w)|+|\lambda_u^{k-M}| |\Pi_u(w)-\mu(w)|
+ |\lambda_u^{k-M}-e^{-\frac{\sigma^2}2 u^2(k-M)}|\, |\mu(w))|\\
&\le c'\alpha^{k-M}\|w\|+c'|\lambda_u|^{k-M}|u| \|w\| + e^{-\frac{\sigma^2}4 u^2(k-M)}
(K-M)|u|^3|\mu(w)|.
\end{align*}
The second term is handled by the  change of variable $v=u\sqrt{k}$ 
$$
\int_{-\delta}^\delta |\lambda_u^{k-M}| \ |u| du \le \int_{-\delta}^\delta e^{-\frac{\sigma^2}4 u^2(k-M)}|u| du \le \frac1{k-M} \int_\RR e^{-\frac{\sigma^2}8v^2}|v|dv \le \frac{C'}{k},
$$
for some constant $C'$. Next, the same change of variable and the dominated convergence theorem shows that the integral of the third term is $
O(k^{-1})|\mu(w)|$. 

Finally, the same change of variable yields
\begin{align*}
\frac{\sqrt{k}}{2\pi}\int_{-\delta}^\delta e^{-iu\ell} e^{-\frac{\sigma^2}2u^2(k-M)} du-\Phi\left(\frac{\ell}{\sqrt{k}}\right)&= \frac{1}{2\pi}\int_{-\delta\sqrt{k}}^{\delta\sqrt{k}}
e^{-\frac{iv\ell}{\sqrt{k}}-\frac{\sigma^2}2 v^2\frac{k-M}{k}}dv-\frac{1}{2\pi}\int_\RR e^{-\frac{iv\ell}{\sqrt{k}}-\frac{\sigma^2v^2}2}dv\\
=O\left(\frac M k\right)=O(k^{-\frac 12})\, .
\end{align*}
Therefore
$$
\sup_{\ell\in\mathbb Z}\left|\frac{\sqrt{k}}{2\pi}\int_{[-\pi,\pi]}e^{-i u \ell} P_u^{k-M}(w)du -\Phi\left(\frac{\ell}{\sqrt{k}}\right)\mu(w)\right| \le \eta_k' \|w\|,
$$
where $\eta_k'=O(k^{-\frac 12})$.
To conclude remark that $P_u^k( H)=P_u^{k-M} P_u^M( H)$, use \eqref{eq1} and Lemma~\ref{PMH}.
\end{proof}
\begin{lemma}\label{lem:Ph>L}
Let $c>0$. 
There exists $K>1$ such that for all $u>0$ small enough, 
\[
 P^k(1_{\{h_k\ge L\}} H)	\le Ke^{-\frac{u L}{\sqrt{k}}}\Vert  H\Vert_{L^1},\label{eqQ2}
\]
uniformly in  $L,M$, in $k\ge M^2$, in $H$ as in Lemma~\ref{PMH}.
\end{lemma}
\begin{proof} 
For all $u>0$, we have
$$		
P^k(1_{\{h_k\ge L\}} H)	\le P^k \left(e^{\frac{u(h_k-L)}{\sqrt{k}}}\, | H|\right)
\le e^{-\frac{Lu}{\sqrt{k}}} 
P
_{-iu/\sqrt{k}}
^k\left(
| H|\right)\, .
$$	
Noticing that $
P_{-iu}(\cdot)=P(e^{uh}\cdot)$. 
Set $w=
 P_{-iu/\sqrt{k}}^M(| H|)=P
 ^M\left(
 e^{\frac u{\sqrt{k}}h_M}
 | H|\right)$.
Noticing that $\frac{M}{\sqrt{k}}\le 1$ and that the $\theta$-H\"older constant of $h_M$ on each $M$-cylinder $D$ satisfies  $|(h_M)_{|D}|_\theta\le \frac{\vert h\vert_\theta\theta^{-M}}{\theta^{-1}-1}$, it follows from
Lemma~\ref{PMH} that $\|w\| \le C \|H\|_1$.

We assume from now on that $|u|<M\delta$. 
By real perturbations in Proposition~\ref{pro:Pu} we get
\begin{align*}
P^k_{-\frac{iu}{\sqrt{k}}}(|H|)=P_{-iu/\sqrt{k}}^{k-M}(w) &\le \left|\lambda_{-iu/\sqrt{k}}^{k-M}\Pi_{-iu/\sqrt{k}}(w)\right|+c'\alpha^{k-M}\|w\|\\
&\le 
(2e^{\frac{3u^2(k-M)}{4k}}+c'\alpha^{k-M})\|w\| \le K \|H\|_1\, ,
\end{align*}
since $|u/\sqrt{k}|\le |u|/M<\delta$.
\end{proof}

For $m,k\in\NN$ we denote the number of partitions with $k$ atoms of a set of $m$ elements by $S(m,k)$, the Stirling number of the second kind.

We now proceed with the proof of the theorem.
\begin{proof}[Proof of Theorem~\ref{THMcvloi}]
It follows from the definition~\eqref{formulaNr} of $\mathcal N_r$
combined with~\eqref{Gry} and from the fact that the balls are cylinders that
\[
\forall (x',y')\in B_r^X(x)\times B_r^Y(y),\quad
\mathcal N_r(x',y')=\sum_{(n,k)\in\mathbb N\times\mathbb Z\, :\, f^n(x')\in B_r^X(x),h_n(x')=k,g^k(y)\in B_r^Y(y)}\delta_{n/n_r}\, .
\]
We set $m_r:=3c_0|\log r|$, where $c_0\ge 1$ is the constant appearing in (II) of Hypothesis~\ref{hypsubshift} (noticing that, since $c_0\ge 1$, this assumption holds also true if we replace $(Y,g,\nu)$ by $(X,f,\mu)$). Changing $c_0$ is necessary, we assume that $c_0\ge\frac{1}{\lambda_X}$.
For any sequence $(k_j)_{j\ge 1}$, we denote the derived sequence $(k_j'=k_j-k_{j-1})_{j\ge 1}$ where we put $k_0=0$.	
	
\underline{Step 1: Moments expressed thanks to the Fourier-perturbed transfer operator}\\
In the following product, we first expand the $\overline m_v$ powers of the sums defining $\N_r([t_{v-1},t_v])$ as $\overline m_v$ sums of a product. Regrouping the indices which are equal we are left with $q_v=1,...,\overline m_v$ distinct indices, which gives after reordering
\begin{align}
\nonumber
\mathcal M_r(\mathbf{t},\mathbf{\overline m})&:=\mathbb E_\rho \left[\prod_{v=1}^K \mathcal N_r(]t_{v-1},t_v])^{\overline m_v} 1_{B_r^X(x)\times B_r^Y(y)} \right] \\
&= 
\sum_{\substack{\mathbf q=(q_1,...,q_K)\\q_v=1,...,\overline m_v}} 
\left(\prod_{v=1}^K S(\overline m_v,q_v)q_v! \right)
A_{n_r;\mathbf {q}
}(x,y) \, ,\label{FormuleMoment}
\end{align}
with
\begin{equation}\label{Anrq}
A_{n_r;\mathbf q}(x,y)
:=\sum_{\mathbf k=(k_1,...,k_q)} \sum_{\boldsymbol{\ell} \in \mathbb Z^q} \mathbb E_\rho \left[  1_{B_r^X(x)\times B_r^Y(y)}\prod_{j=1}^q \left( 1_{B_r^Y(y)}\circ g^{\ell_j}\, 1_{\{h_{k_j}=\ell_j\}}  1_{B_r^X(x)}\circ f^{k_j}\right) \right]
\end{equation}
where we set from now on 
\[q:=q_1+...+q_K
\]
and where the first sum holds over the $\mathbf k=(k_1,...,k_q)$ corresponding to concatenation of $(k^1_1,...,k^1_{q_1})$,...,$(k^v_1,...,k^v_{q_v})$,...,
$(k^K_1,...,k^K_{q_K})$
such that
\[t_{v-1}n_r< k_1^v < \ldots < k_{q_v}^v \leq t_vn_r\, .\]
Recalling that
\[k'_i:=k_i-k_{i-1} \quad\mbox{and}\quad \ell'_i:=\ell_i-\ell_{i-1},\quad k_0=\ell_0=0\, ,
\]
we observe that $A_{n_r;\mathbf q}(x,y)$ can be rewritten
\begin{align*}
&\sum_{\mathbf k
} \sum_{\boldsymbol{\ell} \in \mathbb Z^q} \mathbb E_\rho \left[  1_{B_r^X(x)\times B_r^Y(y)}\prod_{j=1}^q \left( 1_{B_r^Y(y)}\circ g^{\ell_j}\,  1_{\{h_{k_j'}=\ell_j'\}} \circ f^{k_{j-1}} 1_{B_r^X(x)}\circ f^{k_j}\right) \right]
\end{align*}
and so
\begin{equation}\label{FormuleA}
A_{n_r;\mathbf q}(x,y)
=\sum_{
\mathbf k} \sum_{\boldsymbol{\ell} \in \mathbb Z^q} 
\nu \left(\bigcap_{j=0}^q g^{-\ell_j}(B_r^Y(y))\right)\mathbb E _\mu\left[ Q^{(x)}_{k_q',\ell_q'} \cdots Q^{(x)}_{k_1',\ell_1'}(  1_{B_r^X(x)})\right]\, ,
\end{equation}
with 
\[
 Q^{(x)}_{k,\ell} (H)
:=1_{B_r^X(x)} P^k \left(1_{\{h_k=\ell\}}\, H\right).
\]

The strategy of the proof is then to apply inductively Condition (II)
of Hypothesis~\ref{hypsubshift} and Lemma~\ref{PuKH} to say roughly 
\begin{itemize}
\item that $\nu\left(\bigcap_{j=0}^q g^{-\ell_j}(B_r^Y(y))\right)$ behaves as $\left(\nu(B_r^Y(y))\right)^{q+1}$;
\item and that $\mathbb E _\mu\left[ Q^{(x)}_{k_q',\ell_q'} \cdots Q^{(x)}_{k_1',\ell_1'}( 1_{B_r^X(x)})\right]$ behaves as
$\left(\mu(B_r^X(x))\right)^{q+1}\prod_{j=1}^{q}\frac{\Phi\left(\frac{\ell'_j}{\sqrt{k'_j}}\right)}{\sqrt{k'_j}}$\, .
\end{itemize}
Unfortunately this requires some care since the H\"older norm of $ 1_{B_r^X(x)}$ (resp. $ 1_{B_r^Y(y)}$) explodes as $r$ goes to 0. Nevertheless, this will be possible when there are gaps between the indices. In Step 2, we treat the bad situations where there are clusters of indices $k_j$'s (and so lack of gaps). A second difficulty will come from the uniform error (in $\ell'_j$) in the approximation by $\Phi\left(\frac{\ell'_j}{\sqrt{k'_j}}\right)/\sqrt{k'_j}$. To avoid this difficulty we will use Lemma~\ref{lem:Ph>L} to control in Step 3 the contribution of the big values of $\ell'_j$ (and so to restrict the sums over the $\ell_j's$). We will then be able to conclude with the use of Riemann sums (Step 4) and by sum-integral approximations and moment identifications (Step 5).

We say that $k\in\ZZ^q$ has clustering if $k_j'<m_r$ for some $j=1,...,q$.
\\

\noindent\underline{Step 2: Neglectability of clusters of $k_i$'s.\\} 
We will prove that the contribution to $A_{n_r;\mathbf q}(x,y)$ of those $k\in\ZZ^q$ for which there is clustering gives rise to the error term \eqref{error1}.

 For $k_1<\cdots<k_q$ with clustering we denote by $c_1$ the minimal $j\ge 0$ such that $k_{j+1}'<m_r$. The length of the first cluster is $p_1+1$ where $p_1$ is the maximal integer such that $k_{c_1}',\ldots, k_{c_1+p_1}'<m_r$.
We then define inductively the $s$-th cluster (if any) and its length $p_s+1$ by $c_s:=\min\{j\ge c_{s-1}+p_{s-1}\colon k_{j+1}'<m_r^{3^{s-1}}\}$ and $p_s$ is the maximal integer such that $k_{c_s}',\ldots, k_{c_s+p_s}'<m_r^{3^{s-1}}$. Note that the $s$-th cluster starts at the index $c_s$ and ends at the index $c_s+p_s$.
Let $\mathcal{J}_{\mathbf k}=\{1,\ldots,q\}\setminus\bigcup_s \{c_s+1,\ldots,c_s+p_s\}$.
$\{0\}\cup \mathcal{J}_{\mathbf k}$ is the set of indices $j$ which are isolated or where a cluster starts. It determines uniquely the sequences $c_s$ and $p_s$. Note that the existence of a cluster means that $\#\mathcal J_{\mathbf k}<q$.

The sum over $\mathbf k=(k_1,...,k_q)$ in the definition of $A_{n_r;\mathbf{q}}(x,y)$ detailed between~\eqref{Anrq} and~\eqref{FormuleA} can be rewritten as a sum over $\mathcal{J}\subset\{1,\ldots,q\}$ of the sum over the $\mathbf k$'s such that $\mathcal{J}_{\mathbf k}=\mathcal{J}$.

Fix $\mathcal{J}\subset\{1,...,q\}$ with $w:=\# \mathcal{J}<q$.
Consider $\mathbf k$ such that $\mathcal J_{\mathbf k}=\mathcal J$ for which there is a cluster. 
First we have
\[
\mathbb E_\nu \left[\prod_{j=0}^q  1_{B_r^Y(y)}\circ g^{\ell_j}\right] \le \nu \left(\bigcap_{j\in \mathcal{J}\cup\{0\}}g^{-\ell_j}(B_r^Y(y))\right).
\]
Let us consider $j\in\mathcal{J}\cup\{0\}$ a beginning of say the $s$-th cluster. Inside this cluster, $(k'_{j+i})_{i=1,...,p_j}$ can take at most $(m_r^{3^{s-1}})^{p_j}$ values and we know that at time $k_j$ not only we are in the set $B_r^X(x)$ but also we return to it before time $m_r^{3^{s-1}}$. Hence 
\[
\sum_{k_{j+1}..k_{j+p_j}} \sum_{\ell_{j+1}..\ell_{j+p_j}} 1_{\{h_{k_{j+i}'}=\ell_{j+i}'\}}\circ f^{k_{j+i-1}}1_{B_r^X(x)}\circ f^{k_{j+i}}
\le
m_r^{3^sp_j} 1_{\{\tau^f_{B_r^X(x)}\le m_r^{3^s}\}}\circ f^{k_j}.
\]
This means that we can remove all the sums over the indices inside the clusters, provided we insert the above factor in the corresponding place.
This finally gives a contribution to $A_{n_r;\mathbf q}(x,y)$ of those $k$ such that $ \mathcal{J}_{\mathbf k}=\mathcal{J}$ bounded from above by
\begin{equation}\label{truc}
m_r^{3^{q}}\sum_{\mathbf{\overline k}} \sum_{\boldsymbol{\ell}\in\ZZ^w} \nu\left(\bigcap_{j=0}^wg^{-\ell_j}(B_r^Y(y))\right)\mathbb E_\mu\left[\widetilde Q_{w,\bar k_w',\ell_w'}^{(x)}\circ\cdots\circ\widetilde Q_{1,\bar k_1',\ell_1'}^{(x)}(1_{G_{r,0}(x)})\right].
\end{equation}
where the sum is taken over $\mathbf{\overline k}\in\ZZ^w$ which are images by the projection $\mathbf{k}\mapsto (k_j)_{j\in\mathcal{J}}$, and where
\[
\widetilde Q^{(x)}_{i,k_i',\ell_i'} (H)
:=1_{G_{r,i}(x)} P^{k_i'} \left(1_{\{h_{k_i'}=\ell_i'\}}\, H\right).
\]
and  $G_{r,i}(x)= B_r^X(x)\cap \{
\tau^f_{B_r^X(x)}
\le m_r^{3^s}\}$ if the $(i+1)$-th element of $\{0\}\cup \mathcal{J}$ is the beginning of the $s$-th cluster, and $G_{r,i}(x)= B_r^X(x)$ otherwise. Note that, since $c_0\ge 1$, $G_{r,i}(x)$ is a 
union of 
$m_r$-cylinders (if no cluster) or a union of $m_r+m_r^{3^s}$-cylinders, in any case $\bar k_{i+1}'\ge  m_r^{3^{s+1}}\ge (m_r+m_r^{3^s})^2$, so that the lemmas apply. 
We observe that
\begin{equation}\label{formrhoi}
\mathbb E_\mu\left[\widetilde Q_{w,\bar k_w',\ell_w'}^{(x)}\circ\cdots\circ\widetilde Q_{1,\bar k_1',\ell_1'}^{(x)}(1_{G_{r,0}(x)})\right]
=\mathbb E_\mu[1_{G_{r,w}(x)}\rho_w^{\mathbf{\bar k'},\boldsymbol{\ell'}}]\, ,
\end{equation}
with $\rho_0=1$ and defining inductively
\begin{equation}\label{defrhoi}
\forall i=1,...,w, \quad\rho_i^{\bar k_{1..i}, \ell_{1..i}} := P^{\bar k_i'}\left(1_{\{h_{\bar k_i'}=\ell_i'\}} 1_{G_{r,i-1}}\rho_{i-1}^{\bar k_{1..i-1}, \ell_{1..i-1}} \right)\, .
\end{equation}
A computation by induction shows that the norms $\|\log \rho_i^{\bar k_{1..i}, \ell_{1..i}} \|$ are bounded by a constant $C_w$ independent of $\mathbf{\overline k},  \boldsymbol {\ell}$.

We again need to decompose the sum over $\boldsymbol{\ell}\in\mathbb{Z}^w$ subject to clustering or not. Unfortunately clusters may now appear from non consecutive indices and we need to adapt the definition. Given $\boldsymbol{\ell}\in\ZZ^w$, 
for each $i=0,...,w$ we denote by $C_i^{\boldsymbol{\ell}}$ the set of indices $i'$ such that there exists a chain of $\ell_j$'s, pairwise $m_r$-close, joining $\ell_i$ to $\ell_{i'}$.  Next we denote by $\mathcal{I}_{\boldsymbol{\ell}}=\{C_i^{\boldsymbol{\ell}}, i=0,...,w\}$ the set of such clusters.

Fix $\mathcal{I}$ a partition of $\{0,...,w\}$ and consider $\boldsymbol{\ell}$ such that $\mathcal{I}_{\boldsymbol{\ell}}=\mathcal{I}$. Denote by $ \mathcal{I}^*=\{\min C, C\in \mathcal{I}\}
\setminus\{0\}
$ the set of minimal index of each cluster, zero excluded. 
Let $p=\#\mathcal{I^*}=\#\mathcal{I}-1$.
It follows from~(II) of Hypothesis~\ref{hypsubshift} on $g$ that
\begin{align}
\nu\left(\bigcap_{j=0}^w\left(g^{-\ell_j}(B_r^Y(y))\right)\right)
&=
(1+O(\alpha^{m_r}))\prod_{C\in\mathcal I}\nu\left(\bigcap_{i\in C}g^{-\ell_i}(B_r^Y(y))\right)\nonumber \\
&=(1+O(\alpha^{m_r}))\nu(B_r^Y(y))^{\alpha_{\ell}}\nu\left(B_r^Y(y)\cap \{\tau_{B_r^Y(y)}^g<m_r^{3^q}\}\right)^{\beta_{\ell}}\label{majoprodnu}\\
& \le C 
\nu(B_r^Y(y))^{p+1}\nu\left(\left.\tau_{B_r^Y(y)}^g<m_r^{3^q}\right|B_r^Y(y)\right)^{\beta_{\ell}}\, ,
\end{align}
where $\alpha_\ell=\#\{C\in \mathcal I\, :\, \#\{\ell_i,i\in C\}=1\}$
and $\beta_\ell=\#\{C\in \mathcal I\, :\, \#\{\ell_i,i\in C\}>1\}=p+1-\alpha_\ell$. 
It follows from Lemma~\ref{PuKH} that
\[
\mathbb E_\mu\left[1_{G_{r,i}}\rho_i^{\bar k_{1..i},\ell_{1..i}}) \right]
=\mu(G_{r,i}) A_i(\mathbf{\bar k},\boldsymbol{\ell}) \mathbb E_\mu\left[1_{G_{r,i-1}}\rho_{i-1}^{\bar k_{1..i-1},\ell_{1..i-1}}\right]
\]
where, setting $\gamma=\|h\|_\infty$, 
\begin{equation}\label{la1}
 \left|\frac{b_i(\mathbf{k},\boldsymbol{\ell})}{k'_i}:=A_i(\mathbf{k},\boldsymbol{\ell}) - \frac{1}{\sqrt{k_i'}}\Phi\left(\frac{\ell_i'}{\sqrt{k_i'}}\right)\right|
\le\frac{C}{k_i'}\quad\mbox{if }\ |\ell_i'|\le \gamma k_i' 
\end{equation}
and $A_i(\mathbf{k},\boldsymbol{\ell}) =0$ if $|\ell_i'|> \gamma k_i' $. 
Hence an immediate induction gives
\begin{equation}\label{eqrhoprod}
\mathbb E_\mu\left[1_{G_{r,w}}\rho_w^{\mathbf{\bar k},\boldsymbol{\ell}}\right]
\le\left( \prod_{i=0}^w\mu(G_{r,i})\right)
\prod_{i=1}^w A_i(\mathbf{\bar k},\boldsymbol{\ell}).
\end{equation}
Now we fix $\mathbf {\overline k}$ and make the summation over $\boldsymbol{\ell}$ such that $\mathcal{I}_{\boldsymbol{\ell}}=\mathcal{I}$:
\begin{align*}
S_{\mathcal{I}}(\mathbf {\overline k})&:=
\sum_{\boldsymbol{\ell}\colon \mathcal{I}_{\boldsymbol{\ell}}=\mathcal{I}}
\mathbb E_\mu[ 1_{G_{r,w}(x)}\rho_w^{\mathbf{\bar k'},\boldsymbol{\ell'}}]\\
&=
\sum_{\boldsymbol{\ell}\colon \mathcal{I}_{\boldsymbol{\ell}}=\mathcal{I}}
\left( \prod_{j=0}^w\mu(G_{r,j})\right)
\prod_{i=1}^w \frac{1}{\sqrt{\bar k_i'}}\left(\Phi\left(\frac{\ell_i'}{\sqrt{\bar k_i'}}\right)+
\frac{b_i(\mathbf{\bar k},\boldsymbol{\ell})}{\sqrt{\bar k_i'}}\right)1_{\{|\ell_i'|\le \gamma k_i'\}}\\
&=
\sum_{\boldsymbol \ell}
\left( \prod_{j=0}^w\mu(G_{r,j})\right)\prod_{i=1}^w\left(\frac{1}{\sqrt{\bar k_i'}}\left(\Phi\left(\frac{\ell_i'}{\sqrt{\bar k_i'}}\right)+
\frac{b_i(\mathbf{\bar k},\boldsymbol{\ell})}{\sqrt{\bar k_i'}}\right)\right)\, ,
\end{align*}
where the sums are restricted to $ |\ell_i'|\le \gamma \bar k'_i$ for $i=1,...,w$ and $\mathcal{I}_{\boldsymbol{\ell}}=\mathcal{I}$.
During this summation, when $i\in \mathcal{I}^*$ we bound the sum over $\ell_i$ by 
\begin{equation}\label{inI*}
K_0:= \sup_{r<1} \sup_{m_r<k<n_r} \sum_{|\ell| \le \gamma k}  \frac{1}{\sqrt{k}}\left(\Phi\left(\frac{\ell}{\sqrt{k}}\right)+\frac{C}{\sqrt k}\right)<\infty.
\end{equation}
Otherwise when $i\not\in\mathcal{I}^*$, there are at most $2p m_r+1$ choices for $\ell_i$, since $\ell_i$ is close to $\ell_j$ where $j=\min C_i<i$ (hence $\ell_j$ is already fixed). Moreover $\Phi\le 1$ therefore the sum over $\ell_i$ is bounded by 
\begin{equation}\label{notinI*}
((p+1)m_r+1)\frac{1+C}{\sqrt{\bar k_i'}}.
\end{equation}
Therefore
\[
S_{\mathcal{I}}(\mathbf{\overline  k}) = O \left(
\left( \prod_{j=0}^w\mu(G_{r,j})\right) \prod_{i\not\in\mathcal{I}^*}\frac{m_r}{\sqrt{\bar k_i'}}\right).
\]
Putting this last estimate together with~\eqref{formrhoi} and~\eqref{majoprodnu} and summing up over $\mathbf{\bar k}$ gives
\begin{equation}
\begin{split}\label{trucbis}
\eqref{truc} &\le C m_r^{3^{w}} \nu(B_r^Y(y))^{p+1} \left(\prod_{i=0}^w\mu(G_{r,i})\right) (m_r\sqrt{n_r})^{w-p} n_r^{p} \\
&\le C  m_r^{3^{w}} m_r^q \mu(B_r^X(x))\nu(B_r^Y(y))\left(\mu(B_r^X(x))\nu(B_r^Y(y))n_r\right)^p\left(\sqrt{n_r}\mu(B_r^X(x))\right)^{w-p}\\
&\quad\mu\left(\tau^f_{B_r^X(x)}
\le m_r^{3^q}|B_r^X(x)\right) \\
&\le 
C \rho(B_r^Z(x,y))
m_r^{3^q+q} \mu\left(\tau^f_{B_r^X(x)}
\le m_r^{3^q+q}|B_r^X(x)\right) 
\end{split}
\end{equation}
since there were at least one cluster for $k$ ($w<q$) and due to the definition of $n_r=n_r(x, y)$. 
\\

{\it 
We henceforth suppose that there are no cluster of $k_j$'s in the definition of $A_{n_r;\mathbf q}(x,y)$, that is $k_j'>m_r$ for $j=1..q$, hence $w=q$, and $G_{r,i}=B_r^X(x)$ for all $i$.}
\\

\noindent\underline{Step 3: Neglectability of big values of $\ell'_i$.} 
Let $c_r:=q\sqrt{-\log r}$.
The contribution of those $\ell$ such that $\ell_i'> L_i:= c_r \sqrt{k_i'}$ for some $i$ contributes to the error term \eqref{error2}. 

Set $\mathcal{I_\ell}':=\{i=1,...,q\colon \ell_i'>L_i\}$. Fix $\emptyset\neq\mathcal{I}'\subset \mathbb{Z}^q$. We follow the idea in the previous discussion: fix a partition $\mathcal{I}\subset\{0,\ldots,q\}$, define as previously $\mathcal{I}^*$ as the set of starting indices of clusters and $p=\#\mathcal{I}^*$, consider the sum over $\ell$ such that $\mathcal{I}_\ell=\mathcal{I}$ and $\mathcal{I}_\ell'=\mathcal{I}'$.

Recall that $\rho_{i-1}^{ k_{1..i-1},\ell_{1..i-1}}$ has been defined inductively in~\eqref{defrhoi}. 
Set, for the indices $i\in\mathcal{I}'\cap \mathcal{I}^*$ we obtain via Lemma~\ref{lem:Ph>L} 
\[
\begin{split}
\sum_{\ell_i: |\ell_i'|>L_i}
\rho_i^{ k_{1..i},\ell_{1..i}}
& = P^{k_i'}( 1_{\{|h_{k_i'}|> L_i \}}1_{G_{r,i-1}}\rho_{i-1}^{ k_{1..i-1},\ell_{1..i-1}})\\
&\le K e^{-\frac{L_i u}{\sqrt{k_i'}}}\mathbb E_\mu\left[1_{G_{r,i-1}}\rho_{i-1}^{ k_{1..i-1},\ell_{1..i-1}}\right]\, .
\end{split}
\]
Therefore 
\begin{equation}\label{LD}
\sum_{|\ell_i'|>L_i} \mathbb E_\mu\left[1_{G_{r,i}}\rho_i^{k_{1..i},\ell_{1..i}}\right]
\le 
K \mu(G_{r,i})  e^{-u c_r} \mathbb E_\mu\left[1_{G_{r,i-1}}\rho_{i-1}^{k_{1..i-1},\ell_{1..i-1}}\right],
\end{equation}
that is instead of \eqref{inI*} we get the better bound $Ke^{-u c_r}$.

For the indices $i\in\mathcal{I}'\setminus \mathcal{I}^*$, we notice that $\Phi\left(\frac{\ell_i'}{\sqrt{k_i'}}\right)\le \exp(-\frac{\sigma^2c_r^2}2)$, therefore the sum over $\ell_i$ is bounded by 
\[
m_r \left(\frac{e^{-\frac{\sigma^2c_r^2}2}}{\sqrt{k_i'}} +\frac{C}{k_i'}\right),
\]
instead of \eqref{notinI*}.

For the other indices $i\not\in\mathcal{I}'$, we keep the estimates of the previous discussion, 
using using \eqref{inI*} or \eqref{notinI*}, wether $i\in \mathcal{I}^*$ or not. We finally end up with a contribution in
\[
\begin{split}
&\le 
C \mu(B_r^X(x))^{q+1}\nu(B_r^Y(y))^{\#\mathcal{I}} (n_re^{-u c_r})^{\#\mathcal{I}'\cap\mathcal{I}^*}
(m_r\sqrt{n_r}[e^{-\frac{\sigma^2c_r^2}2}+\frac{\log n_r}{\sqrt{n_r}}])^{\#\mathcal{I}'\setminus\mathcal{I}^*} \times\\
&\quad\quad\times
n_r^{\#(\mathcal{I}^*\setminus\mathcal{I}')} (m_r\sqrt{n_r})^{q-\#(\mathcal{I}^*\cup\mathcal{I}')} \\
&\le C \mu(B_r^X(x))\nu(B_r^Y(y)) \left(n_r\mu(B_r^X(x))\nu(B_r^Y(y))\right)^{p} \left(\sqrt{n_r}\mu(B_r^X(x))\right)^{q-p} m_r^q \epsilon_r\\
&\le Cm_r^q\epsilon_r\rho(B_r^Z(x,y))\, ,
\end{split}
\]
with $\epsilon_r:=\max(e^{-uc_r},e^{-\frac12c_r^2}+\frac{\log n_r}{\sqrt{n_r}})$.
\\
\noindent
\underline{Step 4: Reduction to a Riemann sum.}
Recall that $c_r=q\sqrt{-\log r}$.
Thus, up to error terms in \eqref{error1} and \eqref{error2},
in view of \eqref{FormuleMoment},
$A_{n_r,\mathbf q}(x,y)/\rho(B_r^Z(x,y)
)$ behaves as 
\[
\sum_{
	\mathbf k:k_j'>m_r} \mu(B_r^X(x))^{q}
\sum_{\ell_j:|\ell_j'|\le c_r\sqrt{\bar k_j'}}\nu \left(\left.\bigcap_{i=0}^q g^{-\ell_i} (B_r^Y(y)\right|B_r^Y(y))\right)\prod_{j=1}^qA_{j}(k,\ell)\, 
,
\]
with  $\left| A_j(k,\ell)-a_j(k,\ell)\right|\le C/k_j'$ by Lemma~\ref{PuKH}, where $a_j(k,\ell)=\frac{1}{\sqrt{k_j'}}\Phi\left(\frac{\ell_j'}{\sqrt{k_j'}}\right)$.

We aim to replace each $A_j(k,\ell)$ by $a_j(k,\ell)$. Note that they are both bounded by $C/\sqrt{k_j'}$ and their difference is bounded by $C/k_j'$.
Fix a partition $\mathcal{I}$ of $\{1,\ldots,q\}$ as above. Using a telescopic sum we get 
\[
\begin{split}
\left|\prod_{j=1}^qA_{j}(k,\ell)-\prod_{j=1}^q a_{j}(k,\ell)\right| 
&\le\sum_{i'=1}^q 
\prod_{j=1}^{i'-1}A_{j}(k,\ell) \left|A_{i'}(k,\ell)-a_{i'}(k,\ell)\right|\prod_{j=i'+1}^q a_{j}(k,\ell) \\
&\le \sum_{i'=1}^q  \frac{C}{k_{i'}'}\prod_{j\neq i'} \left(a_{j}(k,\ell)+\frac{C}{k_j'}\right) .
\end{split}
\]
We now fix some $i'$ and sum over $\ell$ such that $\mathcal{I}_\ell=\mathcal{I}$ as in the proof of the Step 2,
with the additional condition $|\ell_i'|\le c_r \sqrt{k_i'}$ for all $i$. For the indices $i\neq i'$ we use the same estimates, and for $i=i'$ two cases can happen:

If $i=i'\in\mathcal{I}^*$ we replace the estimate \eqref{inI*} by
\[
\sum_{k_i}\sum_{\ell_i} \frac{C}{k_i'} \le \sum_{k_i} \frac{1+2c_r \sqrt{k_i'}}{k_i'} \le C c_r \sqrt{n_r}.
\]
Otherwise, $i=i'\not\in\mathcal{I}^*$ and we replace the estimate \eqref{notinI*} by
\[
\sum_{k_i}\sum_{\ell_i} \frac{C}{k_i'} \le \sum_{k_i} \frac{1+2m_r}{k_i'} \le C m_r \log{n_r}.
\]
In both cases we gain a factor $\max(m_r\log n_r/n_r,c_r/\sqrt{n_r})$, 
so that the total contribution is dominated by \eqref{error2}.
Thus, writing $e_r:=\eqref{error1}+\eqref{error2}$, we have proved that
\[
\begin{split}
\mathbb E_\rho &\left[ (\mathcal N_r)^m  |B_r^Z(x,y)\right] 
=O\left(e_r\right)+(1+O(e_r))  \mu(B_r^X(x))^{q}
\\
& 
\sum_{
	\mathbf k:k_j'>m_r}
\sum_{\ell_j:|\ell_j'|\le c_r\sqrt{k_j'}}\nu \left(\left.\bigcap_{i=0}^q g^{-\ell_i} (B_r^Y(y))\right| B_r^Y(y)\right)\prod_{j=1}^q \frac{1}{\sqrt{k_j'}}\Phi\left(\frac{\ell_j'}{\sqrt{k_j'}}\right)\, .
\end{split}
\]
\underline{Step 5~: Final step.} It remains to estimate the following quantity
\begin{equation}\label{Final}
 \mu(B_r^X(x))^{q}
\sum_{
	\mathbf k:k_j'>m_r}
\sum_{\ell_j:|\ell_j'|\le c_r\sqrt{k_j'}}\nu \left(\left.\bigcap_{i=0}^q g^{-\ell_i} (B_r^Y(y))\right| B_r^Y(y)\right)\prod_{j=1}^q \frac{1}{\sqrt{k_j'}}\Phi\left(\frac{\ell_j'}{\sqrt{k_j'}}\right)\, ,
\end{equation}
with $c_r=q\sqrt{-\log r}$.
Recall from \eqref{majoprodnu} that
\[
\nu\left(\left.\bigcap_{j=0}^qg^{-\ell_j}B_r^Y(y)\right|B_r^Y(y)\right)
=(1+O(\alpha^{m_r}))\nu(B_r^Y(y))^{
     \#
 \mathcal I_\ell-1}\nu\left(\left.
 \{\tau^g_{B_r^Y(y)}<m_r^{3^q}\}\right|B_r^Y(y)\right)^{\beta_{\ell}}\, ,
\]
with $\mathcal I_{\boldsymbol{ \ell}}$ the partition of $\{0,...,q\}$ consisting in gathering the indices $i$ corresponding to clusters of $\ell_i$'s (as defined in step 2) and with 
$\beta_{\boldsymbol{ \ell}}=\#\{C\in \mathcal I_{\boldsymbol{ \ell}}\, :\, \#\{\ell_i,i\in C\}>1\}
$.
Let $q_0\in\{0,\ldots,q\}$.
Since $\sup_{k>m_r}\sum_{\ell}\frac{\Phi(\ell/\sqrt{k})}{\sqrt{k}}<\infty$,
the contribution of terms with $\#\mathcal I_{\boldsymbol{\ell}}=q_0+1$ and $\beta_\ell=\beta$ is bounded from above by
\begin{align*}
&C\mu(B_r^X(x))^{q}
\nu(B_r^Y(y))^{q_0}\nu\left(\left.
\{\tau_r^g<m_r^{3^q}\}\right|B_r^Y(y)\right)^\beta
\left(\Phi(0)\sum_{k=m_r}^{n_r}k^{-\frac 12}\right)^{q-q_0}
n_r^{q_0} m_r^{3^m\, 1_{\beta\ne 0}}\\
&\le C' (\mu(B_r^X(x))^2)^{\frac{q-q_0}2}
(\mu(B_r^X(x))(\nu(B_r^Y(y)))^{q_0}\nu\left(\left.
\{\tau^g_r<m_r^{3^q}\}\right|B_r^Y(y)\right)^\beta
n_r^{\frac {q-q_0}2+q_0}m_r^{3^m\, 1_{\beta\ne 0}}\\
&\le C'm_r^{3^m\, 1_{\beta\ne 0}}\nu\left(\left.
\{\tau^g_r<m_r^{3^q}\}\right|B_r^Y(y)\right)^\beta\, ,
\end{align*}
due to the definition of $n_r=n_r(x,y)$.
Thus the terms with $\beta_{\boldsymbol{ \ell}}\ge 1$ contribute to the error term \eqref{error1}.

We now consider the terms corresponding to $\beta_{\boldsymbol{ \ell}}=0$, that is we are led to the study of 
\[
\mathcal M_{\mathbf{t}}^{\mathbf{\overline m}}(x,y) := \sum_{\substack{\mathbf q=(q_1,...,q_K)\\q_v=1,...,m'_v}}
\left(\prod_{v=1}^KS(m_v',q_v)q_v!\right)
 A'_{n_r;\mathbf {q}}(x,y)
\]
where
\begin{equation}\label{decompfin}
A'_{n_r;\mathbf {q}}(x,y)
:=\sum_{\boldsymbol{\ell}\,:\,\beta_{\boldsymbol{\ell}}=0}\mu(B_r^X(x))^q(\nu(B_r^Y(y)))^{\# \mathcal{I}_\ell-1}D_{\boldsymbol{\ell}}(\mathbf q)\, ,
\end{equation}
as $n_r/m_r^2\rightarrow +\infty$, i.e. 
as $\mu(B_r^X(x))/m_r\rightarrow +\infty$ and  
$\mu(B_r^X(x))\nu(B_r^Y(y))/(m_r)^2\rightarrow +\infty$, 
with
\[
D_{\mathbf \ell}(\mathbf q)
:= \sum_{\mathbf k
	:k_j'>m_r
}
\prod_{j=1}^q \frac{1}{\sqrt{k_j'}}(\Phi\,  1_{[-c_r,c_r]})\left(\frac{\ell_j'}{\sqrt{k_j'}}\right),
\]
the sum over $\mathbf k$ being still constrained by the $q_v$'s as in~\eqref{Anrq}.

At this step, we can point out the two exteme cases: 
\begin{itemize}
\item[(A)] if $\mu(B_r^X(x))=o(\nu(B_r^Y(y)))$, then the 
terms with $q_0<q$ (i.e. $\#\mathcal I_{\boldsymbol \ell}<q+1$) are negligeable. 
\item[(B)] if $\nu(B_r^X(x))=o(\mu(B_r^Y(y)))$, then the 
terms with $q_0>0$ are negligeable, hence the remaining term is $q_0=0$ and $\beta_{\boldsymbol \ell}=0$ thus $\ell_j=0$ for all $j$.
\end{itemize}
Let us start with the study of the case $q_0=q$ and $\beta_{\boldsymbol \ell}=0$, i.e. $|\ell_j-\ell_{j'}|>m_r$ (the dominating term in Case (A) above). The contribution of these terms is
\begin{align*}
&
(\rho(B_r^Z(x,y)))^{q}
\sum_{\mathbf k}\sum_{\boldsymbol{\ell} }
\prod_{j=1}^q\frac{\Phi\left(\frac{\ell'_j}{\sqrt{k_j'}}\right)}{\sqrt{k_j'}}1_{\{k_j'>m_r\}}1_{\{|\ell_j'|\le c_r\sqrt{k_j'}\}}\\
&\sim (\rho(B_r^Z(x,y)))^{q}   \sum_{\mathbf k}\prod_{j=1}^q\frac{\sum_{\ell'_j}\Phi\left(\frac{\ell'_j}{\sqrt{k_j'}}\right)}{\sqrt{k_j'}}\\
&\sim  (\rho(B_r^Z(x,y)))^{q}   \sum_{\mathbf k}\prod_{j=1}^q\int_{\mathbb R}\Phi(s)\, ds\\
&\sim  (\rho(B_r^Z(x,y)))^{q}\#\{\mathbf k \}\\
&\sim  (\rho(B_r^Z(x,y))n_r)^q \prod_{v=1}^K\frac {(t_v-t_{v-1})^{q_v}}{q_v!}\, .
\end{align*}
A careful analysis would have shown that this equivalence is an equality up to a multiplicative factor $1+O(e^{-\frac{c_r^2}2 })+O({m_r}^{-\frac12})=1+O(|\log r|^{-\frac 12}) $. Furthermore, we recognize the distribution of a standard Poisson process.

Second, we study the contribution of terms of \eqref{decompfin} such that $q_0=0$ and $\beta_{\boldsymbol \ell}=0$, i.e. such that $\ell_j=0$ for all $j$.
This contribution is 
\begin{align}
\nonumber\mu(B_r^X(x))^q\sum_{\mathbf k}\frac{(\Phi(0))^q}{\prod_{j=1}^q\sqrt{k'_j}}&\sim (\sqrt{n_r}\mu(B_r^X(x))\Phi(0))^q\int_{\mathcal E}\prod_{j=1}^{q} (s'_j)^{-\frac 12}\, ds\\
&\sim (\sqrt{n_r}\mu(B_r^X(x)))^q\mathbb E_\mu\left[ \prod_{v=1}^K\frac{(L_{t_v}(0)-L_{t_{v-1}}(0))^{q_v}}{q_v!}\right]\,
,\label{ellnull}
\end{align}
where $\mathcal E$ is the set of $x\in\mathbb R^q$
obtained by concatenation of $(x_1^1,...,x^1_{q_1})$,...,  $(x_1^v,...,x^v_{q_v})$,..., $(x_1^K,...,x^K_{q_K})$
such that $t_{v-1}<x_1^v<...<x_{q_v}^v<t_v$, and where
the last identification follows e.g. from~\cite[Proposition B.1.2]{Maxence_these}.

We return to the general case. Given $q_0\in\{0,...,q\}$, 
we study of the asymptotic as $n_r/m_r^2\rightarrow +\infty$ of the terms of \eqref{decompfin} with $\#\mathcal I_\ell=
q_0+1$ and $\beta_\ell=0$ 
(i.e. clusters in $\mathbf \ell$ corresponds to repetitions of a same value). 

Setting $J_{q\rightarrow q_0}$ for the set of surjections
$\psi:\{0,...,q\}\rightarrow\{0,...,q_0\}$ such that
$\psi(0)=0$, we observe that
\begin{align*}
&\sum_{\mathbf \ell\, :\, \#\mathcal I_\ell=
	q_0+1,\, 
	\beta_\ell=0}D_{\mathbf \ell}(\mathbf{q})
=\frac 1{q_0!}
\sum_{\psi\in J_{q\rightarrow q_0}}
\sum_{\mathbf k:k'_i> m_r
	}\sum_{
(\ell_v)_{v
=1,...,q_0}:|\ell_v-\ell_{v'}|>m_r
}\prod_{j=1}^{q}
\frac{(\Phi\, 1_{[-c_r,c_r]})\left(\frac{\ell_{\psi
			(j)}-\ell_{\psi
			(j-1)}
	}{\sqrt{k'_j}}\right)}{\sqrt{k'_j}}
\\
&\sim \frac 1{q_0!}
\int_{n_r\mathcal E}
\left(
\int_{\mathbb R^{q_0}}
\prod_{j=1}^{q}
\frac{
	\Phi
	\left(\frac{w_{\psi
(j)}-w_{\psi(j-1)}
	}{\sqrt{s'_j}}\right)}{\sqrt{s'_j}}\, dw\right)\, ds
\end{align*}
where we set $\mathcal E$ the set of $s\in\mathbb R^q$
obtained by concatenation of $(s_1^1,...,s^1_{q_1})$,...,  $(s_1^v,...,s^v_{q_v})$,..., $(s_1^K,...,s^K_{q_K})$
such that $t_{v-1}<s_1^v<...<s_{q_v}^v<t_v$, and using again the notation
$s'_j:=s_j-s_{j-1}$ with the conventions $s_0:=0$ and $w_0=0$.
It follows that, if $n_r/m_r^2\rightarrow +\infty$, 
\begin{align*}
&
\sum_{\mathbf \ell\, :\, \#\mathcal I_\ell=q_0+1
	\beta_\ell=0}D_{\mathbf \ell}(\mathbf{q})
\sim \sum_{
\psi\in J_{q\rightarrow q_0}}\frac{n_r^{\frac q2}}{q_0!}
\int_{\mathcal E}
\left(
\int_{\mathbb R^{q_0}}
\prod_{j=1}^{q}
\frac{	\Phi
	\left(\frac{w_{\psi
(j)}-w_{\psi
(j-1)}
	}{\sqrt{n_rs'_j}}\right)}{\sqrt{s'_j}}\, dw\right)\, ds\\
&\frac{n_r^{\frac {q+q_0}2}}{q_0!}\sum_{
	\psi\in J_{q\rightarrow q_0}}
\int_{\mathcal E}
\left(
\int_{\mathbb R^{q_0}}
\prod_{j=1}^{q}
\frac{	\Phi
	\left(\frac{w_{\psi
			(j)}-w_{\psi
(j-1)}
	}{\sqrt{s'_j}}\right)}{\sqrt{s'_j}}\, dw\right)\, ds\\
&\sim
\frac{n_r^{\frac {q+q_0}2}}{q_0!}\sum_{
	\psi\in J_{q\rightarrow q_0}}
\int_{\mathcal E}
\left(
\int_{\mathbb R^{q_0}}\phi_{q,s_1,...,s_q}\left((w_{\psi
(j)})_j\right)\, dw_1...dw_{q_0}\right)ds_1...ds_q\, ,
\end{align*}
where we set $\phi_{q,s_1,...,s_q}$ for the density function of $(B_{s_1},...,B_{s_q})$.

On $\mathcal{E}$ the $s_j$'s are in increasing order. 
For a.e. $s\in \prod_{v=1}^K(t_{v-1},t_v]^{q_v}=:\overline{\mathcal{E}}$ there exists a unique permutation $\pi$, preserving the $K$ blocks, such that $(s_{\pi(j)})_j\in\mathcal{E}$. Applying this change of variables is balanced by substituting $\psi$ by $\psi'=\psi\circ (0\mapsto0,\pi)$. Thus the above quantity is 
\begin{align*}
&=\frac{ n_r^{\frac {q+q_0}2}}{q_0!}
\sum_{	\psi\in J_{q\rightarrow q_0}}\left(\prod_{v=1}^K\frac1{q_v!}\right)
	\int_{\mathbb R^{q_0}}\left(\int_{\overline{\mathcal{E}}}\phi_{q,s_1,...,s_q}\left((w_{\psi
(j)})_j\right)\, ds_1...ds_q\right)dw_1...dw_{q_0}\\
&\sim \frac{ n_r^{\frac {q+q_0}2}}{q_0!}
\sum_{	\psi\in J_{q\rightarrow q_0}}\left(\prod_{v=1}^K\frac1{q_v!}\right)
\int_{\mathbb R^{q_0}}\mathbb E\left[\prod_{j=1}^q (L_{t_{v_q(j)}}-L_{t_{v_q(j)-1}})(w_{\psi	(j)})\right]dw_1...dw_{q_0}\, ,
\end{align*}
where $v_q(j)$ is the smallest integer $v$ such that $j\le q_1+\cdots+q_v$.
So
\[
\sum_{\mathbf \ell\, :\, \#\mathcal I_\ell=q_0+1, 	\beta_\ell=0}D_{ \ell}({\mathbf q})\sim  n_r^{\frac {q+q_0}2} J_r({\mathbf q},q_0)\, ,
\]
with
\[
J_r({\mathbf q},q_0):=
\sum_{	\psi\in J_{q\rightarrow q_0}}
\left(\prod_{v=0}^K\frac 1{q_v!}\right)
\mathbb E\left[
\int_{\mathbb R^{q_0}}
\prod_{j=1}^{q}(L_{t_{v_q(j)}}-L_{t_{v_q(j)-1}})(w_{\psi
(j)})\, dw_1...dw_{q_0}\right]\, .
\]
In view of \eqref{FormuleMoment} and \eqref{decompfin}, we study
\[
\mathcal M_{\mathbf{t}}^{\mathbf{\overline m}}(x,y) 
\sim 
\sum_{\substack{\mathbf q=(q_1,...,q_K)\\q_v=1,...,m'_v}} 
\left(\prod_{v=1}^K S(m_v',q_v)q_v! \right)
A'_{n_r;\mathbf {q}
}(x,y) \, .
\]
As $n_r/m_r^2\rightarrow +\infty$, this quantity is equivalent to
\begin{align*}
&\sum_{\substack{\mathbf q=(q_1,...,q_K)\\q_v=1,...,m'_v}} 
\left(\prod_{v=1}^K S(m_v',q_v) \right)
\sum_{q_0=0}^q \left((n_r\mu(B_r^X(x))^2\right)^{\frac {q}2}\left(n_r\nu(B_r^Y(y)^2)\right)^{\frac{q_0}2}
\, \times \\
&\times
\sum_{	\psi\in J_{q\rightarrow q_0}}
\frac1{q_0!}
\mathbb E\left[
\int_{\mathbb R^{q_0}}
\prod_{j=1}^{q}(L_{t_{v_q(j)}}-L_{t_{v_q(j)-1}})(w_{\psi
	(j)})\, dw_1...dw_{q_0}\right]\, .
\end{align*}
We then conclude by Lemmas~\ref{mompoi} and~\ref{momentPoissonInt}.\\
\end{proof}

\section{Study for the $\mathbb Z$-extension}\label{secZextension}
This section is devoted to the proof of Theorem~\ref{ThmZextension} about the convergence in distribution, as $r\rightarrow 0$, of the return time point process $\widetilde{\mathcal N}_r$ defined in~\eqref{tildeN} of the $\mathbb Z$-extension $\widetilde F$.
The proof of the next result appears as an easy consequence of parts of the proof of Theorem~\ref{THMcvloi}.
\begin{theorem}\label{THMcvloi2}
Assume Hypothesis~\ref{hypsubshift} except (II).
Let $K$ be a positive integer and $\mathbf{\overline m}=(\overline m_1,...,\overline m_K)$ be a $K$-uple of positive integers and let $(t_0=0,t_1,...,t_K)$ be an increasing collection of nonegative real numbers.
There exist a constant $C>0$ and a continuous function $\varepsilon_1$ vanishing at 0 such that, for all $x\in X$,
\begin{align}
\nonumber&\left|
\mathbb E_\mu\left[\left.\prod_{v=1}^K \widetilde{\mathcal N}_r(]t_{v-1},t_{v}])^{\overline m_v}\right|B_r^X(x)\right]- \mathbb E
\left[\prod_{v=1}^K\left(\mathcal Z_{1,0}(t_v)-\mathcal Z_{1,0}(t_{v-1})\right)^{\overline m_v}\right]
\right|\\
\label{error22}&\quad\quad\le \varepsilon_1(\mu(B_r^X(x)))+
\, ,
+C |\log r|^{3^{m}+m}\mu\left(\left.
\tau^f_{B_r^X(x)}\le |\log r|^{3^m}\right|B_r^X(x)\right)\, ,
\end{align}
with $m=|\mathbf{\overline m}|=\overline m_1+\cdots+\overline m_K$, 
and with 
$\mathcal Z_{1,0}$ as in Remark~\ref{Z10}.
\end{theorem}
\begin{proof}
Setting this time $n_r:=(\mu(B_r^X(x)))^{-2}$, 
we follow the scheme of the proof of Theorem~\ref{THMcvloi}
with some simplifications coming from the fact that first $\sum_{\boldsymbol\ell\in\mathbb Z^d}$ therein is replaced by $\boldsymbol{\ell}=\mathbf 0$ and second that $B_r^Y(y)$ disappears. 
As in Step 1 of the proof of Theorem~\ref{THMcvloi}, we observe that
\begin{align}
\nonumber
\mathbb E_\mu \left[\prod_{v=1}^K \widetilde {\mathcal N_r}(]t_{v-1},t_v])^{\overline m_v} 1_{B_r^X(x)} \right]
= 
\sum_{\substack{\mathbf q=(q_1,...,q_K)\\q_v=1,...,\overline m_v}} 
\left(\prod_{v=1}^K S(\overline m_v,q_v)q_v! \right)
A_{n_r;\mathbf {q}
}(x,y) \, ,\label{FormuleMoment2}
\end{align} 
where we denote again $S(m,k)$ for the Stirling number of the second kind (i.e. for the number of partitions with $k$ atoms of a set of $m$ elements), and where we set this time
\begin{equation}\label{Anrq2}
A_{n_r;\mathbf q}(x)
:=\sum_{\mathbf k=(k_1,...,k_q)}  \mathbb E_\mu \left[ \,  1_{B_r^X(x)}\prod_{j=1}^q \left( 1_{\{h_{k_j}=0\}} \,  1_{B_r^X(x)}\circ f^{k_j}\right) \right]
=\sum_{
	\mathbf k} \mathbb E _\mu\left[ Q^{(x)}_{k_q',\ell_q'} \cdots Q^{(x)}_{k_1',\ell_1'}(  1_{B_r^X(x)})\right]\, ,
\end{equation}
with the notations
\[q:=q_1+...+q_K,\quad k'_i:=k_i-k_{i-1},\quad k_0=0\, , \quad
Q^{(x)}_{k,\ell} (H)
:=1_{B_r^X(x)} P^k \left({1}_{\{h_k=\ell\}}\, H\right)\, ,
\]
and where the sum over the $\mathbf k=(k_1,...,k_q)$ corresponds to concatenation of $(k^1_1,...,k^1_{q_1})$,...,$(k^v_1,...,k^v_{q_v})$,...,
$(k^K_1,...,k^K_{q_K})$
such that
\[
t_{v-1}n_r< k_1^v < \ldots < k_{q_v}^v \leq t_vn_r\, .
\]
In Step 2 of the proof of Theorem~\ref{THMcvloi}, since $\boldsymbol{\ell}=\mathbf 0$, $\mathcal I^*=\emptyset$ and $p=0$,~\eqref{truc} and~\eqref{trucbis} ensure that the contribution of clusters of the $m_r$-clusters of $k_j$'s with $\mathcal J_{\mathbf k}= \mathcal J$  is
\begin{align*}
m_r^{3^{q}}&\sum_{\mathbf{\overline k}} \sum_{\boldsymbol{\ell}\in\ZZ^w} \mathbb E_\mu\left[\widetilde Q_{w,\bar k_w',\ell_w'}^{(x)}\circ\cdots\circ\widetilde Q_{1,\bar k_1',\ell_1'}^{(x)}(1_{G_{r,0}(x)})\right]\\
 &\le C m_r^{3^{w}} \left(\prod_{i=0}^w\mu(G_{r,i})\right) (m_r\sqrt{n_r})^{w} \\
&\le C  m_r^{3^{w}} m_r^q \mu(B_r^X(x))\left(\sqrt{n_r}\mu(B_r^X(x))\right)^{w}\mu\left(\tau^f_{B_r^X(x)}\le m_r^{3^q}|B_r^X(x)\right) \\
&\le C \mu(B_r^X(x)) m_r^{3^q+q} \mu\left(\tau^f_{B_r^X(x)}\le m_r^{3^q+q}|B_r^X(x)\right) \, .
\end{align*}
Step 3 of the proof of Theorem~\ref{THMcvloi} disappears (as well as the first part of~\eqref{error2}) since $\ell'_j=0$ for all $j$.\\
As in Step 4 of the proof of Theorem~\ref{THMcvloi}, it follows from
Lemma~\ref{PuKH} that
\begin{align*}
\frac{A_{n_r,\mathbf q}(x)}{\mu(B_r^X(x)
)}&=\mu(B_r^X(x))^{q}
\sum_{	\mathbf k
} 
\left(\left(\prod_{j=1}^q
\frac{\Phi(0)}{\sqrt{k'_j}}
\right)
+\mathcal O\left(\sum_{i'=1}^q  
\frac{
1}{k_{i'}'}\prod_{j\neq i'}\frac 1{\sqrt{k'_j}} 
\right)\right)\\
&= \frac 1{\sqrt{n_r}
^{q}}
\sum_{\mathbf k} 
\prod_{j=1}^q
\frac{\Phi(0)}{\sqrt{k'_j}}
+\mathcal O\left(n_r^{-\frac {1}2}\log(n_r)  \right)
\end{align*}
uniformly in $x\in X$, 
and we conclude by using~\eqref{ellnull} (note that we do not need anymore the control of the probability that $\tau_r^g$ is small used in Step 5 of the proof of Theorem~\ref{THMcvloi}, so that the second part of~\eqref{error1} does not appear here).
\end{proof}
\begin{proof}[Proof of Theorem~\ref{ThmZextension}]
This result follows from Theorem~\ref{THMcvloi2} as  Corollary~\eqref{corCVloi} and Theorem~\ref{thmcvloi2} follow from Theorem~\ref{THMcvloi}.
\end{proof}

\begin{appendix}
\section{Moments of integrals with respect to a Poisson process}\label{append}
The main goal of this appendix is to prove that for any $\alpha,\beta\in[0;1]$ such that $\max(\alpha,\beta)=1$, for all $K$, $m'_1,...,m'_K$ with $m=m'_1+...+m'_K$, $t_1<...<t_K$,
\begin{align}\label{momentK}
&\mathbb E\left[\prod_{v=1}^K\left(\mathcal Z_{\alpha,\beta}(t_{v})-\mathcal Z_{\alpha,\beta}(t_{v-1})\right)^{m'_v}\right]\\
\nonumber&=\sum_{q_0=0}^m \frac{1}{q_0!}\sum_{\substack{\mathbf q=(q_1,...,q_K)\\q_v=1,...,m'_v}}  \left(\prod_{v=1}^{q_1+...+q_K} S(m_v',q_v)\right)\sum_{\psi\in J_{q_1+...+q_K\to q_0}}\alpha^{\frac{q-q_0}2} \beta^{q_0} \\
\nonumber&\quad\mathbb E\left[ \int_{\RR^{q_0}}\prod_{j=1}^{q} (L_{t_{v_q(j)}}(s_{\psi(j)})-L_{t_{v_q(j)-1}}(s_{\psi(j)}))\,  ds_1...ds_{q_0}        \right]\, ,
\end{align}
keeping the notations $S(m,q)$, $v_q(j)$ and $J_{q\rightarrow q_0}$
introduced above in the proof of Theorem~\ref{THMcvloi}.
The case $\alpha=0$ will follow from Lemma~\ref{mompoi}, whereas the case $\alpha>0$ will be studied in Lemma~\ref{momentPoissonInt} (applied with $a:=\sqrt{\alpha}$ and $b:=\beta/\sqrt{\alpha}$).
\begin{lemma}\label{mompoi}
For any nonnegative integer $m$, the moment of order $m$ of a Poisson random variable $\mathcal P_\lambda$ of intensity $\lambda>0$ is
\begin{equation}\label{momentPoisson}
\mathbb E[\mathcal P_{\lambda}^m] =\sum_{q=0}^m S(m,q) \lambda^q\, .
\end{equation}
\end{lemma}
\begin{proof}
Recall that $S(m,q)$ is the number of partitions of a set of $m$ elements in $q$ non-empty subsets. The proof of Lemma~\ref{mompoi} is standard and follows e.g. from the following computation 
\begin{align*}
\mathbb E[e^{t\mathcal P_\lambda}]&=e^{\lambda(e^t-1)}=\sum_{q\ge 0}\frac{(\lambda(e^t-1))^q}{q!}=e^{\lambda(e^t-1)}=\sum_{q\ge 0}\frac{\left(\lambda\sum_{m\ge 1}\frac{t^m}{m!}\right)^q}{q!}\\
&=\sum_{m\ge 0}\left(\sum_{q=0}^m S(m,q)\lambda^q\right)\frac{t^m}{m!}\, .
\end{align*}
\end{proof}

\begin{lemma}\label{A2}
Let $\mathcal P$ be a Poisson process with intensity
$\eta
$ on $\RR$ and $g_j$, $j=1..m$, be bounded integrable functions from $\RR$ to $\RR$. Then
\[
\mathbb E\left[\prod_{j=1}^m\int_{\RR}g_j(s)\, d\mathcal P(s)\right]=\sum_{q=1}^m
\frac1{q!} \sum_{p_i\ge 1\, :\, p_1+...+p_q=m}
\sum_{\chi
}\int_{\RR^q}\prod_{j=1}^mg_j(s_{\chi(j)})\, d\eta(s_1)...d\eta(s_q)\, ,
\]
where the last sum is taken over the set of 
maps $\chi:\{1,...,m\}\rightarrow\{1,...,q\}$ such that $\#\chi^{-1}(\{j\})=p_j$.
\end{lemma}

\begin{proof}
We first claim that
\[
\mathbb E\left[\left(\int_{\RR}g(s)\, d\mathcal P(s)\right)^m\right]=
\sum_{q=1}^m \frac1{q!}\sum_{p_i\ge 1\, :\, p_1+...+p_q=m}\frac{m!}{\prod_{j=1}^q (p_j!)}\int_{\RR^q}\prod_{j=1}^q(g(s_j))^{p_j}\, d\eta(s_1)...d\eta(s_q)\, .
\]
Indeed, using the functional Fourier transform of a Poisson measure \cite{dvj}
\begin{align*}
&\mathbb E\left[e^{i\theta\int_{\RR}g(s)\, d\mathcal P(s)}\right]=\exp\left(
\int_{\RR}(e^{i\theta g(s)}-1)\, d\eta(s)\right)\\
&=1+\sum_{q\ge 1}\frac{\left(
	\int_{\RR}(e^{i\theta g(s)}-1)\, d\eta(s)\right)^q}{q!}\\
&=1+\sum_{q\ge 1}\frac{\int_{\RR^q}(\prod_{i=1}^q(e^{i\theta g(s_i)}-1)\, d\eta(s_1)...d\eta(s_q)}{q!}\\
&=1+\sum_{q\ge 1}\int_{\RR^q}\sum_{p_1\ge 1}...\sum_{p_q\ge 1}\prod_{j=1}^q\frac{(i\theta g(s_j))^{p_j}}{p_j!}\, d\eta(s_1)...d\eta(s_q)\\
&=1+\sum_{m\ge 1}\frac{(i\theta)^{m}}{m!}\sum_{q=1}^m
\sum_{p_i\ge 1\, :\, p_1+...+p_q=m}\frac{m!}{\prod_{j=1}^q (p_j!)}\int_{\RR^q}\prod_{j=1}^q(g(s_j))^{p_j}\, d\eta(s_1)...d\eta(s_q).
\end{align*}
This proves the claim by expanding the exponential in  $\mathbb E\left[e^{i\theta\int_{\RR}g(s)\, d\mathcal P(s)}\right]$ and identifying the $m$-th coefficient.

The lemma follows by identification of the coefficients of $t_1\cdots t_m$ in the following identity, obtained with the claim applied to $\sum_i t_i g_i$ and by direct computation:
\begin{align*}
&\sum_{j_1,...,j_m=1}^m \left(\prod_{u=1}^mt_{j_u}\right)\mathbb E\left[\prod_{u=1}^m\int_{\RR}g_{j_u}(s)\, d\mathcal P(s)\right] =\mathbb E\left[\left(\int_{\RR}\sum_{i=1}^mt_ig_i(s)\, d\mathcal P(s)\right)^m\right]\\
&=\sum_{q\ge1}\sum_{p_i\ge 1\, :\, p_1+...+p_q=m}\frac{m!}{q!\prod_{j=1}^q (p_j!)}\int_{
\RR^q}\prod_{j=1}^q\left(\sum_{i=1}^mt_ig_i(s_j)\right)^{p_j}\, d\eta(s_1)...d\eta(s_q)\, .
\end{align*}
\end{proof}
\begin{lemma}
	\label{momentPoissonInt}
	Let $\mathcal{B}$ be a Brownian motion of variance $\sigma^2$ and $(L_t(\cdot))_t$ its local time. 
Let $\mathcal P$ be a two-sided Poisson process with intensity $b>0$ and let $(\mathcal P'_s)_{s\in \mathbb R}$ be a family of independent homogeneous Poisson processes with intensity $a>0$.
We assume that $\mathcal P$, $\mathcal B$ and $(\mathcal P'_s)$
are mutually independent. 
Let $\mathbf{\overline m}=(\overline m_1,\ldots,\overline m_K)$ be a $K$-uple of positive integers and $\mathbf t=(t_1,\ldots,t_K)$
be a $K$-uple of positive real numbers as in Theorem~\ref{THMcvloi}.
Then
\begin{align*}
\overline{\mathcal{M}}(\mathbf{t},\mathbf{\overline m}):=
&\mathbb E\left[\prod_{v=1}^K\left(\int_{\mathbb R} \mathcal P'_s(L_{t_{v}}(s))-\mathcal P'_s(L_{t_{v-1}}(s)) \, d(\delta_0+\mathcal P)(s)\right)^{\overline m_v}\right]\\
&=\sum_{q_0=0}^m \frac{1}{q_0!}\sum_{\substack{\mathbf q=(q_1,...,q_K)\\q_v=1,...,\overline m_v}}  \left(\prod_{v=1}^q S(\overline m_v,q_v)\right)\sum_{\psi\in J_{q\to q_0}}a^{q} b^{q_0} \\
&\mathbb E\left[ \int_{\RR^{q_0}}\prod_{j=1}^{q} (L_{t_{v_q(j)}}(s_{\psi(j)})-L_{t_{v_q(j)-1}}(s_{\psi(j)}))\,  ds_1...ds_{q_0}        \right]
,
\end{align*}	
with the convention $s_0=0$, $q=q_1+\cdots+q_K$, 
$J_{q\to q_0}$ denotes the set of surjections from $\{1,\ldots,q\}$ to $\{0,\ldots,q_0\}$,
$v_q(j)$ is the smallest integer $v$ such that $j\le q_1+\cdots+q_v$
and $m=\overline m_1+\cdots+\overline m_K$.
\end{lemma}
\begin{proof}
Take $g_j(s)=\mathcal{P}'_s(L_{t_{v(j)}}(s))-\mathcal P'_s(L_{t_{v(j)-1}}(s))$. Expanding the product of the sum then using Lemma~\ref{A2} applied to 
$(g_j)_{j\notin I_0}$ and $E(\cdot|\mathcal{B},(\mathcal{P'}_s))$ give
\begin{align*}
\overline{\mathcal{M}}(\mathbf{t},\mathbf {\overline m})=
&\mathbb E\left[\prod_{j=1}^m\int_{\mathbb R}g_j(s)\, d(\delta_0+\mathcal P)(s)\right]\\
&=\sum_{I_0\subset\{1,...,m\}}\mathbb E\left[\prod_{j_0\in I_0 }g_{j_0}(0)\prod_{j\in\{1,...,m\}\setminus I_0}\int_{\mathbb R}g_j(s)\, d\mathcal P(s)\right]\\
&= 
\sum_{p_0=0}^m \sum_{q_0=0}^{m-p_0}\sum_{p_i\ge 1\, :\, p_1...+p_{q_0}=m-p_0}
\frac{1}{q_0!}\sum_{\chi}\int_{\mathbb{R}^{q_0}
}\mathbb E\left[
\prod_{j=1}^mg_j(s_{\chi(j)})\right]\, b^{q_0} ds_1...ds_{q_0}\, ,
\end{align*}
with the convention $s_0=0$ and 
where the last sum is taken over the set of maps $\chi:\{1,...,m\}\rightarrow\{0,...,q_0\}$ such that $\#\chi^{-1}(\{u\})=p_u$ for $u=0,...,q_0$.
Since the $\mathcal P'_{s_u}$ are a.e independent conditionally to $B$, it follows that
\[
\mathbb E\left[\left.\prod_{j=1}^mg_j(s_{\chi(j)})\right| B\right]
=
\prod_{u=0}^{q_0} K_u(s)
\]
with 
\begin{align*}
K_u(s)&:=\mathbb E\left[\left.\prod_{j\in\chi^{-1}(u)}(\mathcal P'_{s_u}(L_{t_{v(j)}(s_u)})-\mathcal P'_{s_u}(L_{t_{v(j-1)}}(s_u))\right|B\right]\\
&=\prod_{w=1}^K
\mathbb E\left[\left.(\mathcal P'_{s_u}(L_{t_{v}}(s_u)-L_{t_{v-1}}(s_u)))^{\#\{j\in\chi^{-1}(u):v(j)=w\}}\right|B\right]\\
&=\prod_{w=1}^K \sum_{z_{u,w}=0}^{m^\chi_{u,w}} S(m^\chi_{u,w},z_{u,w}) (b.(L_{t_{w}}(s_u)-L_{t_{w-1}}(s_u)))^{z_{u,w}}\, .
\end{align*}
with $m^\chi_{u,w}:={\#\{j=1,...,m\colon \chi(j)=u,v(j)=w\}}$.
So
\begin{align*}
&\int_{\RR^{q_0}}\mathbb E\left[
\prod_{j=1}^mg_j(s_{\chi(j)})\right]\, ds_1...ds_{q_0}\, .
\\  &=\sum_{D}
\left(\prod_{u,w} S(m^\chi_{u,w},z_{u,w})\right) a^{|D|} H(D)\, ,
\end{align*}
where the sum is over all the matrices $Z'=(z'_{u,w})_{u=0,...,q_0,w=1,...,K}$ such that $\forall u,w, 0 \le z_{u,w}\le m^\chi_{u,w}$
with $|Z'|={\sum_{u,w}z_{u,w}}$ and
\begin{equation}\label{H}
H(Z'):=\mathbb E\left[\int_{\RR^{q_0}}\prod_{u=0}^{q_0}\left(\prod_{w=1}^K  (L_{t_w}(s_u)-L_{t_{w-1}}(s_u))^{z_{u,w}}\right)\,  ds_1...ds_{q_0}        \right]\, .
\end{equation}
Hence, we have proved that
\begin{align*}
&\overline{\mathcal{M}}(\mathbf{t},\mathbf{\overline m})=\\
&=
\sum_{q_0=0}^m\frac{1}{q_0!}b^{q_0}
\sum_{p_0\ge0,p_i\ge 1\, :\, p_0+...+p_{q_0}=m}
\sum_{\chi}
\sum_{Z'}
\left(\prod_{u,w} S(m^\chi_{u,w},z_{u,w})\right) a^{|Z'|} H(Z')\\
&=
\sum_{q_0=0}^m\frac{1}{q_0!}
\sum_{\substack{\mathbf q=(q_1,...,q_K)\\q_v=1,...,\overline m_v}} a^q b^{q_0}
\sum_{Z'}  H(Z')
\left[\sum_{\chi}
\left(\prod_{u,w} S(m^\chi_{u,w},z_{u,w})\right)\right],
\end{align*}
where in the last line the matrices $Z'$ are such that $ \sum_{u=0}^{q_0}z_{u,w}=q_w$ for all $w=1..K$, and the surjection $\chi$ is such that $m_{u,w}^\chi\ge z_{u,w}$.

Let $Z'$ be such that $\prod_{u,w} S(m^\chi_{u,w},z_{u,w})\ne 0$. Then, for every $u,w$ such that 
$m^\chi_{u,w}\ge 1$, we also have  $z_{u,w}\ge 1$.
This ensures that both $q_w:=\sum_{u=0}^{q_0}z_{u,w}$
and $\sum_{w=1}^{K}z_{u,w}$ are non null.

Denote by $C(Z')$ the term into brackets above.
We claim that $C(Z')$ is the number of colored partitions in the following sense:
a partition of $\{1,\ldots,m\}$ which refines the partition in successive temporal blocks of size $m_v'$ and to each element of the partition is assigned a value in $0..q_0$, with the constraint that for each $u,w$ there are $z_{u,w}$ elements of the partition in the $w$th temporal block giving the value $u$. 

Indeed, choosing $\chi $ assigns to each integer in $\{1,...,m\}$ a unique value in $\{0,...,q_0\}$.
For each temporal block $w$ and each value $u$, the set of integers $j$ in the $w$th temporal block such that $\chi(j)=u$ has cardinality  $m_{u,w}^\chi $. We partition it into $z_{u,w}$ atoms. There are $S(m_{u,w}^\chi,z_{u,w})$ possibilities. This defines uniquely a partition with the prescribed property, and there are $C(Z')$ possibilities.

Another method to construct such a colored partition is to first partition each temporal $v$ block in $q_v$ atoms. There are $S(m_v',q_v)$ possibilities.
This refined partition has $q$ elements. 
Any surjection $\psi\in J_{q,q_0}$ assigns to the $j$th atom (ordered by their minimal element) of the partition a value $\psi(j)$. Let $Z^\psi$ be the matrix with entries $z_{u,w}^\psi$ equal to the number of atoms in the $w$th temporal block with value $u$, that is the number of $j$ such that $v_q(j)=w$ and $\psi(j)=u$. We restrict to those $\psi$ such that $Z^\psi=Z'$.
Note that there are $\prod_{v=1}^K S(\overline m_v,q_v)\#\{\psi\in J_{q,q_0}\colon Z^\psi=Z'\}$ possibilities.
Therefore 
\[
\overline{\mathcal{M}}(\mathbf{t},\mathbf{\overline m})
=
\sum_{q_0=0}^m\frac{1}{q_0!}
\sum_{\substack{\mathbf q=(q_1,...,q_K)\\q_v=1,...,m'_v}} a^qb^{q_0}\left(\prod_{v=1}^q S(\overline m_v,q_v)\right)
\sum_{Z'} \sum_{\psi\in J_{q\to q_0}\colon Z^\psi=Z'} H(Z')\, .
\]
\end{proof}

Finally, we give for completeness a convergence result with the moments method under assumptions slightly weaker than usual. The subtlety is due to the fact that the moments converge in restriction to a good set that may depend on the exponent. Given $x\in \RR^K$ and $m\in\NN^K$ we let $x^m=\prod_{v=1}^K x_v^{m_v}$.
\begin{lemma}\label{carl+}
Let $W$, $X_n$'s be $\RR^K$ valued random variables, such that 
\begin{itemize}
	\item There exists a sequence $(\Omega_n)_{n}$ of measurable sets such that
	$P(\Omega_n)\to1$ as $n\to\infty$
	\item For every $m\in \NN^K$, $E(X_n^m 1_{\Omega_n})\to E(W^m)$ as $n\to\infty$
	\item $W$ satisfies the Carleman's criterion $\sum_{m=0}^\infty E(\|W\|^m)^{-1/m}=\infty$.
\end{itemize}
Then $X_n$ converges in distribution to $W$.
\end{lemma}
\begin{proof}
It follows from the classical Carleman's criterion (see e.g.~\cite{Schmudgen}) that
the linear combinations of $(X_n1_{\Omega_n})_n$ converge in distribution to those of $W$, which implies the convergence in distribution of $(X_n1_{\Omega_n})_n$ to $W$.
Furthermore $X_n-X_n1_{\Omega_n}$ converges in distribution to 0.
We conclude (using e.g. the Slutzky lemma) that $(X_n)_n$ converges in distribution to $W$.	
\end{proof}

\end{appendix}

{\bf Acknowledgements~:} F. P\`ene conduced this work within the framework of the Henri Lebesgue Center ANR-11-LABX-0020-01 and is supported by the ANR project GALS (ANR-23-CE40-0001).


 \end{document}